\documentclass{amsproc}

\usepackage[english]{babel}
\usepackage{sidecap}

\newtheorem{theorem}{Theorem}[section]
\newtheorem{lemma}[theorem]{Lemma}
\newtheorem{proposition}[theorem]{Proposition}

\theoremstyle{definition}
\newtheorem{definition}[theorem]{Definition}
\newtheorem{example}[theorem]{Example}

\newtheorem{corollary}[theorem]{Corollary}

\theoremstyle{remark}
\newtheorem{remark}[theorem]{Remark}
\DeclareMathOperator\diag{diag}

\numberwithin{equation}{section}

\begin{document}

\title[Spectral Tetris]{The fundamentals of Spectral Tetris frame constructions}

\author{Peter G. Casazza and Lindsey M. Woodland}
\address{Department of Mathematics, University of Missouri, Columbia,
MO 65211}
\curraddr{}
\email{casazzap@missouri.edu; lmwvh4@mail.missouri.edu}
\thanks{The authors were supported by:  NSF DMS 1008183;  NSF ATD 1042701; AFOSR DGE51: FA9550-11-1-0245}

\subjclass[2000]{Primary }


\date{}

\begin{abstract}
In a landmark paper \cite{CFMWZ09}, Casazza, Fickus, Mixon, Wang and Zhou
introduced a fundamental method for constructing unit norm tight frames,
which they called {\it Spectral Tetris}.  This was a significant advancement for finite frame theory - especially constructions of finite
frames.  This paper then generated a vast amount of literature as Spectral Tetris was steadily developed, refined, and generalized until today we have a complete picture of what are the broad applications as well as the limitations of Spectral Tetris.  In this paper, we will put this vast body
of literature into a coherent theory. 
\end{abstract}

\maketitle

\section{Introduction}

Hilbert space frames were introduced by Duffin and Schaeffer in \cite{DS}
while studying deep questions in non-harmonic Fourier series.  Today they
have broad application to problems in pure mathematics, applied mathematics,
engineering, medicine and much more.  A fundamental problem for applications
of frames is to construct frames with the necessary properties for the application.
This can often be very difficult if not impossible in practice.

In a landmark paper, \cite{CFMWZ09}, in 2009 Casazza, Fickus, Mixon, Wang and Zhou
introduced a fundamental technique for the construction of unit norm tight frames
which they called {\it Spectral Tetris}.  This technique was a significant advancement for
finite frame theory since prior to that we had no broad general methods for constructing
finite frames.  In some specific cases, ad-hoc techniques were developed to produce
frames for certain applications, but in most cases the theory relied on {\it existence
proofs} for knowing that required frames existed.  This at least allowed researchers
to spend their time trying to produce the needed frames for applications with the
knowledge that they existed.  But the ad-hoc methods were time consuming and
did not produce any new very general classes of finite frames for applications.

It was immediately clear that the paper \cite{CFMWZ09} was just the beginning of something with much broader applications to the construction of finite frames.
This caused a flurry of activity around Spectral Tetris as it was steadily developed, refined, and generalized until today we have a complete picture of what finite frames and fusion frames it can and cannot produce.  It was known right from the beginning that Spectral Tetris cannot construct all unit norm tight frames. But today, we do have necessary and sufficient conditions for Spectral Tetris to work in a broad variety of generalizations and situations.  In this paper we will put this vast quantity of literature into a coherent theory so that researchers will be able to quickly tell if these methods will work for the problems they are working on.

The outline of the present paper is as follows. In Section (2) and Section (3) we provide necessary background information on Hilbert space frames and fusion frames, respectively. In Section (4) we discuss the construction techniques available before Spectral Tetris was discovered. Section (5) provides the basics of Spectral Tetris and the original Spectral Tetris construction method for unit norm, tight frames with eigenvalues greater than or equal to two. Spectral Tetris is then adapted in section (6) to construct unit norm, tight frames with positive eigenvalues. Next, we generalize Spectral Tetris in Section (7) to allow for the construction of non-tight frames. We conclude our Spectral Tetris construction of frames in Section (8) by providing a generalized Spectral Tetris construction method for frames and necessary and sufficient conditions for when this method is applicable. Next, we consider construction methods for fusion frames and in Section (9) we give a brief introduction to Spectral Tetris fusion frame constructions. Then in Section (10), we see how the original Spectral Tetris construction method is generalized to construct 2-sparse, equidimensional, unit-weighted fusion frames. This method is further generalized in Section (11) to construct unit weighted Spectral Tetris fusion frames. Our final construction method for fusion frames occurs in Section (12) where we provide a generalized Spectral Tetris fusion frame construction algorithm as well as necessary and sufficient conditions for when this method is applicable. Lastly, in Section (13) we give concluding remarks.


\section{Hilbert Space Frames}

We now introduce the basics of finite frame theory.

\begin{definition}
A family of vectors $\{f_n\}_{n=1}^N$ in an $M$-dimensional Hilbert space $\mathcal{H}_M$ is a frame if there are constants $0 < A \leq B < \infty$ so that for all $x \in \mathcal{H}_M$,
\[
A \| x \|^2 \leq \sum_{n=1}^N |\langle x,f_n\rangle|^2 \leq B \|x\|^2,
\]
where $A$ and $B$ are the lower and upper frame bounds, respectively.

\begin{enumerate}
\item In the finite dimensional setting, a frame is simply a spanning set of vectors in the Hilbert space.

\item The {\it optimal lower frame bound and optimal upper frame bound}, denoted $A_{op}$ and $B_{op}$ respectively, are the largest lower frame bound and the smallest upper frame bound, respectively.

\item If $A = B$ is possible, then $\{f_n\}_{n=1}^N$ is a {\it tight frame}. Moreover, if $A = B = 1$ is possible, then $\{f_n\}_{n=1}^N$ is a {\it Parseval frame}.

\item If there is a constant $c$ so that $\|f_n\| = c$ for all $n = 1,\ldots,N$ then $\{f_n\}_{n=1}^N$ is an {\it equal norm frame}. Moreover, if $c = 1$ then $\{f_n\}_{n=1}^N$ is a {\it unit norm frame}.  

\item $\{\langle x,f_n\rangle \}_{n = 1}^N$ are called the {\it frame coefficients} of the vector $x\in \mathcal{H}_M$ with respect to frame $\{f_n\}_{n=1}^N$.

\item We will refer to a unit norm, tight frame as a UNTF.

\end{enumerate}
\end{definition}

If $\{f_n\}_{n=1}^N$ is a frame for $\mathcal{H}_M$, then the {\it analysis operator} of the frame is the operator $T : \mathcal{H}_M \to \ell_2(N)$ given by
\[
T(x) = \{ \langle x,f_n\rangle \}_{n=1}^N
\]
and the {\it synthesis operator} is the adjoint operator, $T^*$, which satisfies
\[
T^{*}\left( \{a_n\}_{n=1}^N \right) = \sum_{n=1}^N a_n f_n.
\]
The {\it frame operator} is the positive, self-adjoint, invertible operator $S = T^* T$ on $\mathcal{H}_M$ and satisfies
\[
S(x) = T^*T(x) = \sum_{n=1}^N \langle x,f_n\rangle f_n.
\]

That is, $\{f_n\}_{n=1}^N$ is a frame if and only if there are constants $0<A\leq B<\infty$ such that its frame operator $S$ satisfies $AI \leq S \leq BI$ where $I$ is the identity on $\mathcal{H}_M$.

We say that a frame has a certain spectrum or certain eigenvalues if its frame operator $S$ has this spectrum or respectively these eigenvalues. Note that the spectrum of a frame operator $S$ is positive and real. Also, the smallest and largest eigenvalues of a frame operator $S$ coincide with the optimal lower and upper frame bounds, respectively. For any frame with spectrum $\{\lambda_m\}_{m=1}^M$, the sum of its eigenvalues, counting multiplicities, equals the sum of the squares of the norms of its vectors:
\[\sum_{m=1}^M\lambda_m = \sum_{n=1}^N ||f_n||^2.\]
This quantity will be exactly the number of vectors $N$ when we work with unit norm frames. 

\begin{corollary}
If $\{f_n\}_{n=1}^N$ is a UNTF for $\mathcal{H}_M$ then the frame bound will be $c=\frac{N}{M}$.
\end{corollary}

To each frame we can associate the matrix of its synthesis operator, where the columns correspond to the frame vectors represented against an orthonormal basis for $\mathcal{H}_M$. Note, that any $M \times N$ matrix with $N \geq M$ and which has rank $M$, could be representative of the synthesis matrix of a frame. However, this arbitrary matrix may not have many helpful properties in applications since its only property is that it spans $\mathcal{H}_M$. As we will see in the next theorem, if we represent the frame vectors against the eigenbasis of its frame operator $S$ then the synthesis matrix will possess some very nice properties. 

\begin{theorem}\label{thm2.3}\cite{CK}
Let $T : \mathcal{H}_M \to \ell_2(N)$ be a linear operator, let $\{e_m\}_{m=1}^M$ be an orthonormal basis for $\mathcal{H}_M$, and let $\{\lambda_m\}_{m=1}^M$ be a sequence of positive numbers. Let $A$ denote the $M \times N$ matrix representation of $T^{*}$ with respect to $\{e_m\}_{m=1}^M$ and the standard basis $\{\hat{e}_n \}_{n=1}^N$ of $\ell_2(N)$. Then the following conditions are equivalent.
\begin{enumerate}
\item $\{T^* \hat{e}_n \}_{n=1}^N$ forms a frame for $\mathcal{H}_M$ whose frame operator has eigenvectors $\{e_m\}_{m=1}^M$ and associated eigenvalues $\{\lambda_m\}_{m=1}^M$.
\item The rows of $A$ are orthogonal and the $m$-th row square sums to $\lambda_m$.
\item The columns of $A$ form a frame for $\ell_2(N)$ and $A A^* = \diag(\lambda_1,\ldots,\lambda_M)$, where $AA^*$ represents the frame operator and ``diag" is the diagonal operator with diagonal values $\{\lambda_m\}_{m=1}^M$.
\end{enumerate}
\end{theorem}

The preceding theorem implies that to construct a frame, which is useful in applications, one only needs to find a matrix $A$ with nonzero orthogonal rows. Then the column vectors of $A$ will form a frame, represented against the eigenbasis of its frame operator $S$, and for which the square sum of the rows are the eigenvalues of $S$ and so the square sum of the columns are the squared norms of the frame vectors. Furthermore, the rows must all square sum to the same number for the frame to be tight and the columns must all square sum to the same value for the frame to be equal norm. Because of this, in the present paper we will assume all frames are represented against the eigenbasis of their frame operator. Theorem \ref{thm2.3} also justifies calling such a matrix a {\it frame matrix} or just a {\it frame} and hence we will use the term {\it frame} interchangeably to mean a frame or a frame matrix.

Our current goal in the present paper is to develop methods for constructing a frame, or more specifically the synthesis matrix of a frame. As we have just seen, when working with the synthesis matrix of a frame one can easily find the eigenvalues of the frame operator and the norms of the frame vectors, which is key in classifying different types of frames. We will first look at known construction methods for a synthesis matrix with prescribed properties. In particular, one well known theorem used for frame construction illustrates how to construct a Parseval frame from the knowledge of an existing Parseval frame. Consider the following construction: for $N \geq M$, given an $N \times N$ unitary matrix, if we select any $M$ rows from this matrix then the column vectors from these rows form a Parseval frame for $\mathcal{H}_M$. Moreover, the leftover set of $N-M$ rows, also has the property that its $N$ columns form a Parseval frame for $\mathcal{H}_{N-M}$. The next theorem, known as Naimark's Theorem, utilizes this type of operation and is one of the most fundamental results in frame theory.

\begin{theorem}[Naimark's Theorem]\cite{CK}\label{naimark}
Let $F=\{f_n \}_{n=1}^N$ be a frame for $\mathcal{H}_M$ with analysis operator $T$, let $\{e_n\}_{n=1}^N$ be the standard basis of $\ell_2\left(N\right)$, and let $P:\ell_2\left(N\right) \rightarrow \ell_2\left(N\right)$ be the orthogonal projection onto $\mbox{range}\left(T\right)$. Then the following conditions are equivalent:
\begin{enumerate}
\item $\{f_n\}_{n=1}^N$ is a Parseval frame for $\mathcal{H}_M$.
\item For all $n=1,\dots,N$, we have $Pe_n=Tf_n$.
\item There exist $\psi_1,\dots, \psi_N \in \mathcal{H}_{N-M}$ such that $\{ f_n \oplus \psi_n \}_{n=1}^N$ is an orthogonal basis of $\mathcal{H}_N$.
\end{enumerate}

Moreover, if (3) holds, then $\{\psi_n\}_{n=1}^N$ is a Parseval frame for $\mathcal{H}_{N-M}$. If $\{\psi'_n\}_{n=1}^N$ is another Parseval frame as in (3), then there exists a unique linear operator $L$ on $\mathcal{H}_{N-M}$ such that $L\psi_n=\psi_n',$ for all $i=1,\dots,N$, and $L$ is unitary.
\end{theorem}

Explicitly, we call $\{\psi_n\}_{n=1}^N$ the {\it Naimark Complement} of $F$. 

With all of the necessary definitions from finite frame theory needed for the present paper complete, we refer the interested reader to \cite{CK,Ch} for a more in-depth study of finite frames.


\section{Fusion Frames}\label{ffsec} 

Due to the redundancy, flexibility and stability of a frame, frame theory has proven to be a powerful area of research with applications to a wide array of fields, including signal processing, noise and erasure reduction, compressed sensing, sampling theory, data quantization, quantum measurements, coding, image processing, wireless communications, time-frequency analysis, speech recognition, bio-imaging, and much more. The reader is referred to \cite{CK} and references therein for further information regarding these applications and more.

In particular, a frame $\{f_n\}_{n=1}^N$ for $\mathcal{H}_M$ with frame operator $S$ possesses the property of perfect reconstruction. That is, for every $x \in \mathcal{H}_M$, we have \[x=\sum_{n=1}^N\langle x, f_n \rangle S^{-1} f_n= \sum_{n=1}^N\langle x, S^{-1}f_n \rangle f_n.\] However, with the recent advances in technology, it is not always the case that a signal (or data in general) can be handled by a single processing system, i.e, by a single frame, and because of this, the necessity for new theories needed to be established. 

Today, across numerous disciplines, scientists utilize vast amounts of data obtained from various networks which need to be analyzed at a central processor. However, due to low communication bandwidth and limited transit/computing power at each single node in the network, the data may not be able to be computed at one centralized processing system. Hence there has been a fundamental shift from centralized information processing to distributed processing, where network management is distributed and the reliability of individual links is less critical. Here the data processing is performed in two stages: (1) local processing at neighboring nodes, followed by (2) the integration of locally processed data streams at a central processor.

An example of distributed processing involves wireless sensor networks, which can provide cost-effective and reliable surveillance. Consider a large number of inexpensive, small sensors dispersed throughout an area in order to collect data about the area or to keep surveillance. Due to practical and economic factors such as the topography of the land, limited signal processing power, low communication bandwidth, or short battery life, the sensors are not capable of transmitting their information to one central processor. Therefore, the sensors need to be deployed in smaller clusters, where in each cluster there is one higher powered sensor which collects all of the information from the signals in its cluster and then transmits this information to a central processor. In this two-stage model, information is first gathered locally in each cluster, then processed more globally at a central station. A similar local-global processing principle is also applicable to modeling the human visual cortex \cite{RJ}.

Mathematically, if we view each cluster as a subspace of a larger space, then as illustrated in this example and in general, large sensor networks can be seen as a redundant collection of sub-networks forming a set of subspaces in some space. More explicitly, given data and a collection of subspaces, first project the data onto the subspaces, then process the data within each subspace. Next, combine or {\it fuse} all of the locally processed information. Now we know information about how the data interacts with the whole space and not just with each individual subspace. To relate this back to our wireless sensor example, we have that the decomposition of the given data into the subspaces coincides with the local clusters of sensors gathering information locally. The fusion of the locally processed information then coincides with the larger powered sensors in each cluster transmitting their information to a central processor. The distributed fusion, which models the reconstruction of the data, also enables an error analysis of resilience against erasures. This is, however, only possible if the data is decomposed in a redundant way, which forces the subspaces to be redundant.

This concept of a frame-like collection of subspaces is known as a {\it fusion frame} and provides a suitable mathematical framework to design and analyze two-stage processing. Because of this, fusion frame theory in used in applications where two-stage (local/global) analysis is required, with applications situated in areas which require distributed processing, such as: distributed sensing, parallel processing, packet encoding and optimal packings \cite{CK}. In finite frame theory, a fusion frame is a spanning collection of subspaces, which were first studied in \cite{CK04}, and have been further analyzed in \cite{CCHKP10,framesofsubspaces, CK, CKL}.

Fusion frame theory is a generalization of frame theory. To illustrate this connection, recall that for a given frame $\{f_n\}_{n=1}^N$, its frame operator can be viewed in the following manner:
\[ Sx=T^*Tx=\sum_{n=1}^N\langle x,f_n \rangle f_n = \sum_{n=1}^N ||f_n||^2\langle x,\frac{f_n}{||f_n||}\rangle \frac{f_n}{||f_n||}.\]

Notice that $S$ is the sum of rank one projections each with weight given by the square norm of the respective frame vector. Generalizing this idea to consider weighted projections of arbitrary rank yields the definition of a fusion frame. Explicitly a fusion frame is as follows:

\begin{definition}
Let $\{ W_i\}_{i=1}^D$ be a family of subspaces in $\mathcal{H}_M$, and let $\{ w_i\}_{i=1}^D \subseteq \mathbb{R}^{+}$ be a family of weights. Then $\{\left( W_i,w_i\right)\}_{i=1}^D$ is a {\em fusion frame} for $\mathcal{H}_M$ if there exist constants $0 < A\leq B < \infty$ such that
\[A\|x\|_2^2 \leq \sum_{i=1}^D w_i^2 \|P_i\left( x \right)\|_2^2 \leq B\|x\|_2^2 \mbox{ {\ \ }for all } x \in \mathcal{H}_M,\]
where $P_i$ denotes the orthogonal projection of $\mathcal{H}_M$ onto $W_i$ for each $i \in \{1,\dots,D\}$. 
\begin{enumerate}
\item The constants $A$ and $B$ are called the {\em lower fusion frame bound and upper fusion frame bound}, respectively.  
\item The largest lower fusion frame bound and the smallest upper fusion frame bound are called the {\em optimal lower fusion frame bound and optimal upper fusion frame bound}, respectively.
\item If $A=B$ is possible then the family $\{\left( W_i,w_i \right)\}_{i=1}^D$ is called a {\em tight fusion frame}.  Moreover, if $A=B=1$ is possible then the family $\{\left( W_i,w_i \right)\}_{i=1}^D$ is called a {\em Parseval fusion frame}. 
\item If each subspace has unit weight, $w_i=1$ for all $i=1,\dots, D$, then the family $\{\left( W_i, w_i \right)\}_{i=1}^D$ is simply denoted $\left(W_i \right)_{i=1}^D$ and is called a {\em unit weighted fusion frame}.
\item Lastly, the {\it fusion frame operator} $\widetilde{S}:\mathcal{H}_M \rightarrow \mathcal{H}_M$ defined by $\widetilde{S}x= \sum_{i=1}^Dw_i^2P_i\left(x\right)$ for all $x \in \mathcal{H}_M$ is a positive, self-adjoint, invertible operator, where $P_i$ is the orthogonal projection of $\mathcal{H}_M$ onto $W_i$.
\item The $\{W_i\}_{i=1}^D$ are called the {\em fusion frame subspaces}.
\end{enumerate}
\end{definition}

Recall that when considering a conventional frame a signal can be represented by a collection of scalars, which measure the amplitudes of the projections of the signal onto the frame vectors. Generalizing this idea to a fusion frame we now represent a signal by a collection of vectors via the projections of the signal onto the subspaces of the fusion frame. Explicitly, given a fusion frame $\{\left(W_i,w_i\right)\}_{i=1}^D$, any signal $x \in \mathcal{H}_M$ can be represented as $x= \sum_{i=1}^D w_i^2\widetilde{S}^{-1}\left(P_i\left(x\right)\right)$. In particular, if $\{\left(W_i,w_i\right)\}_{i=1}^D$ is a tight fusion frame with tight fusion frame bound $A$, then the fusion frame operator is a multiple of the identity and becomes $\widetilde{S}=AI$ yielding the representation $x=A^{-1}\sum_{i=1}^Dw_i^2\left(P_i\left(x\right)\right)$ for any signal $x \in \mathcal{H}_M$. In a two-stage data processing setup, these orthogonal projections serve as locally processed data, which can be combined to reconstruct the signal of interest. 

We have seen that in two-stage processing, a signal can be reconstructed via a fusion frame. However, due to sensor failures, buffer over flows, added noise or subspace perturbations during the two stage processing, some information about the signal could be lost or corrupted. One might ask, how can a fusion frame reconstruct a signal when these problems are present? Clearly, redundancy between the subspaces helps to add resilience against erasures (or lost data), as mentioned earlier; but what about other issues that could arise when a signal is being processed. Redundancy between these subspaces may not be sufficient to manage these issues and typically extra structure on the fusion frame is required, such as prescribing the subspace dimensions or prescribing the fusion frame operator.  In particular, \cite{KPC, KPCL} show that in order to minimize the mean-squared error in the linear minimum mean-squared error estimation of a random vector from its fusion frame measurements in white noise, the fusion frame needs to be Parseval or tight. Also to provide maximal robustness against erasures of one fusion frame subspace the fusion frame subspaces must also be equidimensional.

Within two stage processing, further issues could potentially arise due to economic factors which limit the available computing power and bandwidth for data processing. And because of this we need to be able to construct a fusion frame that enables signal decomposition with a minimal number of additions and multiplications. Seeing these numerous potential constraints on our data processing capabilities, we now have  motivation to determine the existence and construction of fusion frames that not only have a desired fusion frame operator or subspace dimensions, but also possess some degree of {\it sparsity} in order to reduce computational cost. 

Next we will define and discuss what it means for a fusion frame to be {\it sparse}; but before we do this we will first discuss sparse frames. By a {\it sparse} frame we mean that the frame vectors have few non-zero coefficients with respect to a fixed orthonormal basis. Explicitly,

\begin{definition} Given a fixed orthonormal basis of $\mathcal{H}_M$, a vector in $\mathcal{H}_M$ which can be represented by only $0 \leq k \leq M$ basis elements, is called {\it $k$-sparse}.
\end{definition}

Recently, sparsity has become an important concept in various areas of applied mathematics, computer science, and electrical engineering. Many types of signals  possess sparse representations when choosing a suitable basis or frame. As such, their reconstruction simplifies and in general these signals can be recovered from few measurements using $\ell_1$ minimization techniques. Since fusion frames generalize the structure of a frame, it is natural to question if sparse representations in fusion frames posses similar properties as sparse representations in frames. In particular, do sparse fusion frames allow for precise signal reconstruction when using only an under determined set of equations? The answer to this question is yes, which leads to a further question: how can such a sparse fusion frame be constructed?

Since a fusion frame is a collection of subspaces which could potentially have large dimensions and/or a spectrum with a wide range, then the computational complexity for recovering a signal via a fusion frame greatly increases from that of a conventional frame. Thus to alleviate some of this computation and to speed up processing time, we would like our subspaces to be as sparse as possible. In particular, if each subspace was spanned by a collection of sparse vectors with respect to a fixed orthonormal basis for $\mathcal{H}_M$, then this would greatly help with these issues. We can achieve sparse fusion frames if each vector of such a subspace basis is $k$-sparse with small $k$. We now make this definition clear.

\begin{definition}
A fusion frame $\{\left(W_i,v_i\right)\}_{i\in I}$ for an indexing set $I$, is $k$-sparse with respect to an orthonormal basis $\{e_j\}_{j=1}^M$ for $\mathcal{H}_M$ if each subspace $W_i$ is spanned by an orthonormal basis $\{e_{ij}\}_{j=1}^{m_i}$ so that for each $j = 1, 2,\dots,m_i$, we have $e_{ij} \in \mbox{ span}\{e_\ell\}_{\ell \in J}$ and $|J| \leq k$.
\end{definition}

Since fusion frames are necessary in two-stage data processing, which is used in a wide array of fields, we would like to be able to construct sparse fusion frames satisfying desired properties. This way researchers can implement our construction techniques to help them with their problems in the areas of distributed processing.  Similar to working with the synthesis matrix of a frame, we will work with the matrix representative of a fusion frame, as described in the following theorem.

\begin{theorem} \cite{framesofsubspaces}\label{ffmatrix} The following are equivalent:
\begin{enumerate}
\item $\{\left( W_i, w_i \right)\}_{i=1}^D$ is a fusion frame for $\mathcal{H}_M$ with lower and upper fusion frame bounds $A$ and $B$, respectively.
\item There exists an orthonormal basis $\{e_{ij}\}_{j=1}^{d_i}$ for $W_i$, for all $i=1,\dots,D$, so that the matrix $B$ with column vectors $e_{ij}$ for $i \in \{1,\dots, D\}$ and $j \in \{1,\dots, d_i\}$ satisfies:
\begin{enumerate}
\item The rows are orthogonal and
\item the square sums of the rows lie between $A$ and $B$.
\end{enumerate}
\end{enumerate}
\end{theorem} 
 
Similar to our discussion of conventional frames, the smallest and largest eigenvalues of the fusion frame operator correspond to the optimal smallest and largest fusion frame bounds, $A$ and $B$. Moreover, the square sum of the rows of the fusion frame matrix, as described in Theorem \ref{ffmatrix}, yield the eigenvalues of the fusion frame operator, and hence if all of the rows of such a matrix square sum to the same value, then we have a tight fusion frame. In the present paper, when we discuss the eigenvalues of a fusion frame, we specifically mean the eigenvalues of its fusion frame operator.

In Section \ref{ffsection} and after, we provide easily implementable algorithms for the construction of sparse fusion frames with desired properties. This culminates to the most generalized algorithm for fusion frame constructions which is seen in Section \ref{gffsec}. Prior to the development of these construction algorithms, there were other methods for constructing fusion frames. However these methods first require the knowledge of a given fusion frame. In particular, two general ways to construct a fusion frame from a given fusion frame are the {\it Spatial Complement Method} and the {\it Naimark Complement Method}, which we now explain.

Given a fusion frame, taking its spatial complement is a natural way of generating a new fusion frame. In order to see this, we first need the definition of an {\em orthogonal fusion frame to a given fusion frame}.

\begin{definition}\cite{CCHKP10}
Let $\{\left( W_i, w_i\right)\}_{i=1}^D$ be a fusion frame for $\mathcal{H}_M$. If the family $\{\left( W_i^{\perp}, w_i\right)\}_{i=1}^D$, where $W_i^{\perp}$ is the orthogonal complement of $W_i$, is also a fusion frame, then we call $\{\left( W_i^{\perp}, w_i \right)\}_{i=1}^D$ the {\it orthogonal fusion frame} to $\{\left( W_i, w_i\right)\}_{i=1}^D$.
\end{definition}

With this definition now clear, we proceed with the spatial complement method for construction. 

\begin{theorem}[Spatial Complement Theorem]\cite{CCHKP10}\label{spatialcomp}
Let $\{\left( W_i, w_i\right)\}_{i=1}^D$ be a fusion frame for $\mathcal{H}_M$ with optimal fusion frame bounds $0 < A\leq B<\infty$ such that $\sum_{i=1}^D w_i^2<\infty$. Then the following conditions are equivalent:
\begin{enumerate}
\item $\bigcap_{i=1}^D W_i=\{0\}$.
\item $B < \sum_{i=1}^D w_i^2$.
\item The family $\{\left( W_i^{\perp}, w_i\right)\}_{i=1}^D$ is a fusion frame for $\mathcal{H}_M$ with optimal fusion frame bounds $\sum_{i=1}^D w_i^2-B$ and $\sum_{i=1}^D w_i^2-A$.
\end{enumerate}
\end{theorem}

Theorem \ref{spatialcomp} provides an easy method for determining a new fusion frame from a given fusion frame, and also yields information regarding the fusion frame bounds for the new fusion frame. However, this theorem is only applicable if you want the orthogonal fusion frame of a given fusion frame. 

Another fusion frame construction method, which requires a given fusion frame is called the {\it Naimark Complement Method}. Recall, Naimark's Theorem for frames, as stated in Theorem \ref{naimark}. Since fusion frames are a generalization of frames then we can define the Naimark complement of a fusion frame through the use of the Naimark complement of a conventional frame. Consider the following relationship between frames and fusion frames. Let $\{\left(W_i,w_i\right)\}_{i=1}^D$ be a fusion frame for $\mathcal{H}_M$ with frame operator $\widetilde{S}$. Let $\left(\psi_{i,j}\right)_{j=1}^{d_i}$ be an orthonormal basis for $W_i$ for $i=1,\dots, D$ and let $T$ be the analysis operator for the family $\left(W_i, \psi_{i,j}\right)$, then we have the following equivalence:
\[\widetilde{S}x=\sum_{i=1}^Dw_i^2\left(P_i\left(x\right)\right)=\sum_{i=1}^D\sum_{j=1}^{d_i}w_i^2\langle x, \psi_{i,j}\rangle \psi_{i,j} =\]\[ \sum_{i=1}^D\sum_{j=1}^{d_i} \langle x, w_i \psi_{i,j}\rangle w_i \psi_{i,j}=T^*Tx=Sx.\]

Thus we see that the fusion frame operator and the frame operator are equivalent in this scenario. Thus every fusion frame arises from a conventional frame partitioned into equal-norm, orthogonal sets. Using this relationship, we can define the Naimark complement of a fusion frame via the Naimark complement of a frame. 

\begin{definition}
Let $\{\left(W_i,w_i\right)\}_{i=1}^D$ be a Parseval fusion frame for $\mathcal{H}_M$. Choose orthonormal bases $\left(\psi_{i,j}\right)_{j=1}^{d_i}$ for $W_i$, making $\{w_i \psi_{i,j}\}_{i=1,j=1}^{D,d_i}$ a Parseval frame for $\mathcal{H}_M$. By Theorem \ref{naimark}, $\{w_i \psi_{i,j}\}_{i=1,j=1}^{D,d_i}$ has a Naimark complement Parseval frame $\{\psi'_{i,j}\}_{i=1,j=1}^{D,d_i}$ for $\mathcal{H}_{D-M}$. The {\it Naimark Complement fusion frame} of $\{\left(W_i,w_i\right)\}_{i=1}^D$ is given by \[\left\{\left( W_i',\sqrt{1-w_i^2}\right)\right\}_{i=1}^D,\] which is a Parseval fusion frame for $\mathcal{H}_{\sum_{i=1}^Dd_i-D}$, where $W_i':= \mbox{ span }\left(\{\psi_{i,j}'\}_{j=1}^{d_i}\right)$.
\end{definition}

Notice that the choice of the orthonormal bases for the subspaces $W_i$ of a fusion  frame will alter the corresponding Naimark complement fusion frame. However, it is shown in \cite{CFMPS} that all choices yield unitarily equivalent Naimark complement fusion frames in the sense that there is a unitary operator mapping the corresponding fusion frame subspaces onto one another. Now with the knowledge of what a Naimark complement fusion frame is, the next theorem provides properties for when one exists. 

\begin{theorem} [Naimark Complement Method]\cite{CCHKP10}\label{naimarkff}
Let $\{\left( W_i, w_i\right)\}_{i=1}^D$ be a Parseval fusion frame for $\mathcal{H}_M$ with $0< w_i< 1$, for all $i=1, \dots, D$. Then there exists a Hilbert space $\mathcal{K} \subseteq \mathcal{H}_M$ and a Parseval fusion frame $\{\left( W_i', \sqrt{1- w_i^2}\right)\}_{i=1}^D$ for $\mathcal{K} \ominus \mathcal{H}_M$ with dim$W_i' = $dim$W_i$ for all $i=1,\dots,D$.
\end{theorem}

Theorem \ref{naimarkff} provides a nice method for determining when a Naimark complement fusion frame exists and gives an exact description of this new fusion frame and its subspace dimensions. Notice, however that the Naimark complement method for fusion frames is only applicable to Parseval fusion frames, much like how the Naimark Theorem for conventional frames was only applicable to Parseval frames. Both the spatial complement method and the Naimark complement method for constructing fusion frames are useful in some applications where the corresponding complement fusion frame has certain desired properties that the original fusion frame may lack. However, if no given fusion frame is known, then neither theorem is useful for construction. In Section \ref{ffsection} through Section 12 of the present paper, we will see numerous Spectral Tetris construction techniques for sparse fusion frames with prescribed properties. These construction algorithms were the first of their kind and before their creation the only methods for fusion frame constructions were the previous ones mentioned. Before we explicitly describe our Spectral Tetris fusion frame construction algorithms we will develop the theory behind Spectral Tetris frame constructions. We will eventually see that the Spectral Tetris frame constructions lend nicely to the fusion frame constructions.


\section{Before Spectral Tetris}

Before Spectral Tetris, the field relied on {\it existence theorems} to tell if certain frames must exist.  But these theorems did not give the exact vectors which formed the required frames.  The main results here are from \cite{CFKT, CL,CL1}.  As it turns out, these results existed in the literature prior to these papers (known as the {\it Schur-Horn Theorem}) but were in a form that was not recognized earlier.  However, \cite{CL,CL1} certainly provide the best proofs available for the Schur-Horn Theorem. 

\begin{theorem}[Schur-Horn Theorem] \label{shthm}\cite{CL1}
Let $S$ be a positive, self-adjoint operator on $\mathcal{H}_M$, and let $\lambda_1
\ge \lambda_2 \ge \cdots \ge \lambda_M>0$ be the eigenvalues of $S$.
Further, let $N\ge M$, and let $a_1 \ge a_2 \ge \cdots \ge a_N$ be positive
real numbers.  The following are equivalent:

(1)  There exists a frame $\{f_n\}_{n=1}^N$ for $\mathcal{H}_M$ having frame operator
$S$ and satisfying $\|f_n\|=a_n$ for all $n=1,2,\ldots,N$.

(2)  For every $j\le k \leq M$ we have
\[ \sum_{j=1}^ka_j^2 \le \sum_{j=1}^k \lambda_j\mbox{ and } \sum_{j=1}^Na_j^2 = 
\sum_{j=1}^N \lambda_j.\]
\end{theorem}

Based on Theorem \ref{shthm}, the existence of a frame $\{f_n\}_{n=1}^N$ for $\mathcal{H}_M$ with vector norms $\{a_n\}_{n=1}^N$ and spectrum $\{\lambda_m\}_{m=1}^M$ is characterized by conditions (1) and (2). In particular, properties (1) and (2) state that an $N$-element frame in $\mathcal{H}_M$ with lengths $\{a_n\}_{n=1}^N$ and eigenvalues $\{\lambda_m\}_{m=1}^M$ of the frame operator exists if and only if the sequence of eigenvalues $\{\lambda_m\}_{m=1}^M$ {\it majorizes} the sequence of square norms $\{a_n^2\}_{n=1}^N$. Explicitly this means:

\begin{definition}\label{majorization}
After arranging both sequences, $\{a_n\}_{n=1}^N$ and $\{\lambda_m\}_{m=1}^M$, in non-increasing order, if $\sum_{i=1}^na_i^2 \leq \sum_{i=1}^n \lambda_i$ for every $n=1,\dots,M$ and $\sum_{i=1}^Na_i^2=\sum_{i=1}^M\lambda_i$, then $\{\lambda_m\}_{m=1}^M$ {\it majorizes} $\{a_n^2\}_{n=1}^N$. We denote this by $\{\lambda_m\}_{m=1}^M \succeq   \{a_n^2\}_{n=1}^N$. Moreover, if $M \neq N$ then add zeroes to the end of the shorter sequence to make them the same length. 
\end{definition}

From Theorem \ref{shthm} it follows that,

\begin{corollary} \cite{CFKT}\label{shcor}
For every $N\ge M$ and every invertible, positive, self-adjoint operator $S$ on $\mathcal{H}_M$
there exists an equal norm frame for $\mathcal{H}_M$ with $N$-elements and
frame operator $S$.  In particular, there exists an equal norm Parseval frame with
$N$-elements in $\mathcal{H}_M$ for every $N \geq M$.
\end{corollary} 

Both Theorem \ref{shthm} and Corollary \ref{shcor} are helpful in the sense that they guaranteed that the frames we were searching for must exist, but did not give any help in actually finding the required frames.  Spectral Tetris provided the first major construction technique for a wide variety of frames and fusion frames.


\section{Spectral Tetris Frame Constructions: The Basics of Spectral Tetris}\label{basicst}

Spectral Tetris was introduced in ``Constructing tight fusion frames," \cite{CFMWZ09}, as a method for constructing sparse, unit norm, tight frames and sparse, unit weighted, tight fusion frames via a quick and easy to use algorithm. We start with an example which illustrates the basics of Spectral Tetris for unit norm, tight frames (UNTFs). Note that we will call any frame constructed via Spectral Tetris, a {\it Spectral Tetris frame}.

Before we begin our example, let us go over a few necessary facts for construction. Recall, that in order to construct an $N$-element UNTF in $\mathcal{H}_M$, we will construct an $M \times N$ synthesis matrix having the following properties:
\begin{enumerate}
\item The columns square sum to one, to obtain unit norm vectors.
\item The rows are orthogonal, which is equivalent to the frame operator, $S$, being a diagonal $M \times M$ matrix.
\item The rows have constant norm, to obtain tightness, meaning that $S=cI$ for some constant $c$, where $I$ is the $M \times M$ identity matrix. 
\end{enumerate}

\begin{remark} Since we will be constructing $N$-element UNTFs in $\mathcal{H}_M$, recall that the frame bound will be $c=\frac{N}{M}$. \end{remark}

Also, before construction of a frame is possible, we must first ensure that such a frame exists by checking that the spectrum of the frame majorizes the square vector norms of the frame. However, this is not the only constraint. For Spectral Tetris to work, we also require that the frame has redundancy of at least 2, that is $N \geq 2M$, where $N$ is the number of frame elements and $M$ is the dimension of the Hilbert space. For a UNTF, since our unique eigenvalue is $\frac{N}{M}$, we see that this is equivalent to the requirement that the eigenvalue of the frame is greater than or equal to 2.  

The main idea of Spectral Tetris is to iteratively construct a synthesis matrix, $T^*$, for a UNTF one to two vectors at a time, which satisfies properties (1) and (2) at each step and gets closer to and eventually satisfies property (3) when complete. When it is necessary to build two vectors at a time throughout the Spectral Tetris process, we will utilize the following key $2 \times 2$ matrix as a building block for our construction.

Spectral Tetris relies on the existence of $2 \times 2$ matrices $A\left(x\right)$, for given $0\leq x \leq 2$, such that:
\begin{enumerate}
\item the columns of $A\left(x\right)$ square sum to $1$,
\item $A\left(x\right)$ has orthogonal rows,
\item the square sum of the first row is $x$.
\end{enumerate} 
These properties combined are equivalent to $$A\left(x\right)A^*\left(x\right)=\left[\begin{array}{cc}
x&0\\
0&2-x
\end{array}\right].$$

A matrix which satisfies these properties and which is used as a building block in Spectral Tetris is: 

\begin{eqnarray}
\label{eqn:a}  
A\left(x\right)=\left[\begin{array}{cc}
\sqrt{\frac{x}{2}}& \sqrt{\frac{x}{2}}\\ 
\sqrt{1-\frac{x}{2}}&-\sqrt{1-\frac{x}{2}}
\end{array}\right].
\end{eqnarray}

We start with an example of how the Spectral Tetris algorithm works. 

\begin{example}\label{stcex}
We would like to use Spectral Tetris to construct a sparse, unit norm, tight frame with 11 elements in $\mathcal{H}_4$, so our tight frame bound will be $\frac{11}{4}$.

To do this we will create a $4 \times 11$ matrix $T^*$, which satisfies the following conditions:
\begin{enumerate}
\item The columns square sum to $1$.
\item $T^*$ has orthogonal rows.
\item The rows square sum to $\frac{11}{4}$.
\item $S=T^*T=\frac{11}{4}I$.
\end{enumerate}

Note that (4) follows if (1), (2) and (3) are all satisfied. First note that, although it is clear for UNTFs, our sequence of eigenvalues $\{\lambda_m\}_{m=1}^4=\{\frac{11}{4},\frac{11}{4},\frac{11}{4},\frac{11}{4}\}$ majorizes the sequence of our square norms $\{a_n^2\}_{n=1}^{11}=\{1,1,1,1,1,1,1,1,1,1,1\}$, which, in general, is necessary for such a frame to exist. 
 
Define $t_{i,j}$ to be the entry in the $i^{th}$ row and $j^{th}$ column of $T^*$. With an empty $4 \times 11$ matrix, we start at $t_{1,1}$ and work our way left to right to fill out the matrix. By requirement (1), we need the square sum of column one to be 1 and by requirement (2) we need the square sum of row one to be $\frac{11}{4} \geq 1$. Hence, we will start by being greedy and put the maximum weight of 1 in $t_{1,1}$. This forces the rest of the entries in column 1 to be zero, from requirement (1). We get:
$$T^*=\left[\begin{array}{ccccccccccc}
1&\cdot&\cdot&\cdot&\cdot&\cdot&\cdot&\cdot&\cdot&\cdot&\cdot\\
0&\cdot&\cdot&\cdot&\cdot&\cdot&\cdot&\cdot&\cdot&\cdot&\cdot\\
0&\cdot&\cdot&\cdot&\cdot&\cdot&\cdot&\cdot&\cdot&\cdot&\cdot\\
0&\cdot&\cdot&\cdot&\cdot&\cdot&\cdot&\cdot&\cdot&\cdot&\cdot
\end{array} \right].$$

Next, since row one needs to square sum to $\frac{11}{4}$, by (3), and we only have a total weight of 1 in row one, then we need to add $\frac{11}{4}-1=\frac{7}{4}=1+\frac{3}{4}\geq 1$ more weight to row one. So we will again be greedy and add another 1 in $t_{1,2}$. This forces the rest of the entries in column 2 to be zero, by (1). Also note that we have a total square sum of 2 in row one. We get:
$$T^*=\left[\begin{array}{ccccccccccc}
1&1&\cdot&\cdot&\cdot&\cdot&\cdot&\cdot&\cdot&\cdot&\cdot\\
0&0&\cdot&\cdot&\cdot&\cdot&\cdot&\cdot&\cdot&\cdot&\cdot\\
0&0&\cdot&\cdot&\cdot&\cdot&\cdot&\cdot&\cdot&\cdot&\cdot\\
0&0&\cdot&\cdot&\cdot&\cdot&\cdot&\cdot&\cdot&\cdot&\cdot
\end{array} \right].$$

In order to have a total square sum of $\frac{11}{4}$ in the first row, we need to add a total of $\frac{11}{4}-2=\frac{3}{4}<1$ more weight. If the remaining unknown entries are chosen so that $T^*$ has orthogonal rows, then $S$ will be a diagonal matrix. Currently, the diagonal entries of $S$ are mostly unknowns, having the form $\{2+?,\cdot,\cdot,\cdot\}$. Therefore we need a way to add $\frac{3}{4}$ more weight in the first row without compromising the orthogonality of the rows of $T^*$ nor the normality of its columns. That is, if we get ``greedy" and try to add $\sqrt{\frac{3}{4}}$ to position $t_{1,3}$ then the rest of row one must be zero, yielding:

$$T^*=\left[\begin{array}{ccccccccccc}
1&1&\sqrt{\frac{3}{4}}&0&0&0&0&0&0&0&0\\
0&0&\cdot&\cdot&\cdot&\cdot&\cdot&\cdot&\cdot&\cdot&\cdot\\
0&0&\cdot&\cdot&\cdot&\cdot&\cdot&\cdot&\cdot&\cdot&\cdot\\
0&0&\cdot&\cdot&\cdot&\cdot&\cdot&\cdot&\cdot&\cdot&\cdot
\end{array} \right].$$

In order for column three to square sum to one, at least one of the entries $t_{2,3}, t_{3,3}$ or $t_{4,3}$ is non-zero. But then, it is impossible for the rows to be orthogonal and thus we cannot proceed. Hence, we need to instead add two columns of information in attempts to satisfy these conditions. The key idea is to utilize our $2 \times 2$ building block, $A\left(x\right)$, as defined at (\ref{eqn:a}).

We define the third and fourth columns of $T^*$ according to such a matrix $A(x)$, where $x=\frac{11}{4}-2=\frac{3}{4}$. Notice that by doing this, column three and column four now square sum to one  within the first two rows, hence the rest of the unknown entries in these two columns will be zero. We get:
$$T^*=\left[\begin{array}{ccccccccccc}
1&1&\sqrt{\frac{3}{8}}&\sqrt{\frac{3}{8}}&\cdot&\cdot&\cdot&\cdot&\cdot&\cdot&\cdot\\
0&0&\sqrt{\frac{5}{8}}&-\sqrt{\frac{5}{8}}&\cdot&\cdot&\cdot&\cdot&\cdot&\cdot&\cdot\\
0&0&0&0&\cdot&\cdot&\cdot&\cdot&\cdot&\cdot&\cdot\\
0&0&0&0&\cdot&\cdot&\cdot&\cdot&\cdot&\cdot&\cdot
\end{array} \right].$$

The diagonal entries of $T^*$ are now $\{\frac{11}{4}, \frac{5}{4}+?,\cdot,\cdot\}$. The first row of $T^*$, and equivalently the first diagonal entry of $S$, now have sufficient weight and so its remaining entries are set to zero. The second row, however, is currently falling short by $\frac{11}{4}-\left(\left(\sqrt{\frac{5}{8}}\right)^2 + \left(-\sqrt{\frac{5}{8}}\right)^2\right)=\frac{6}{4}=1+\frac{2}{4}$. Since $1+\frac{2}{4} \geq 1$, we can be greedy and add a weight of 1 in $t_{2,5}$. Hence, column five becomes $e_2$. Next, with a weight of $\frac{2}{4}<1$ left to add to row two we utilize our $2 \times 2$ building block $A\left(x\right)$, with $x=\frac{2}{4}$. Adding this $2 \times 2$ block in columns six and seven yields sufficient weight in these columns and hence we finish these two columns with zeros. We get: 

$$T^*=\left[\begin{array}{ccccccccccc}
1&1&\sqrt{\frac{3}{8}}&\sqrt{\frac{3}{8}}&0&0&0&0&0&0&0\\
0&0&\sqrt{\frac{5}{8}}&-\sqrt{\frac{5}{8}}&1&\sqrt{\frac{2}{8}}&\sqrt{\frac{2}{8}}&0&0&0&0\\
0&0&0&0&0&\sqrt{\frac{6}{8}}&-\sqrt{\frac{2}{8}}&\cdot&\cdot&\cdot&\cdot\\
0&0&0&0&0&0&0&\cdot&\cdot&\cdot&\cdot
\end{array} \right].$$

The diagonal entries of $T^*$ are now $\{\frac{11}{4}, \frac{11}{4},\frac{6}{4}+?,\cdot\}$, where the third diagonal entry, and equivalently the third row, are falling short by $\frac{11}{4}-\frac{6}{4}=\frac{5}{4}=1+\frac{1}{4}$. Since $1+\frac{1}{4} \geq 1$, then we take the eighth column of $T^*$ to be $e_3$. We will complete our matrix following these same strategies, by letting the ninth and tenth columns arise from $A\left(\frac{1}{4}\right)$, and making the final column $e_4$, yielding the desired UNTF:

$$T^*=\left[\begin{array}{ccccccccccc}
1&1&\sqrt{\frac{3}{8}}&\sqrt{\frac{3}{8}}&0&0&0&0&0&0&0\\
0&0&\sqrt{\frac{5}{8}}&-\sqrt{\frac{5}{8}}&1&\sqrt{\frac{2}{8}}&\sqrt{\frac{2}{8}}&0&0&0&0\\
0&0&0&0&0&\sqrt{\frac{6}{8}}&-\sqrt{\frac{2}{8}}&1&\sqrt{\frac{7}{8}}&\sqrt{\frac{7}{8}}&0\\
0&0&0&0&0&0&0&0&\sqrt{\frac{7}{8}}&-\sqrt{\frac{7}{8}}&1
\end{array} \right].$$

\end{example}

In this construction, column vectors are either introduced one at a time, such as columns $1,2,5,8,$ and $11$, or in pairs, such as columns $\{3,4\}, \{6,7\},$ and $\{9,10\}$. Each singleton contributes a value of 1 to a particular diagonal entry of $T^*$, while each pair spreads two units of weight over two entries. Overall, we have formed a flat spectrum, $\{\frac{11}{4},\frac{11}{4},\frac{11}{4},\frac{11}{4}\}$, from blocks of area one or two. This construction is reminiscent of the game Tetris, as we fill in blocks of mixed area to obtain a flat spectrum.

From this construction it is clear that $T^*$ has a large number of zero entries, which happens to be a nice property that the Spectral Tetris construction produces. This is known as sparsity and recall that a vector in $\mathcal{H}_M$ which can be represented by only $0 \leq k \leq M$ basis elements, is called {\it $k$-sparse}.
Hence, in Example \ref{stcex}, column one of $T^*$ is 1-sparse and column three is 2-sparse. The sparsity of $T^*$ in Example \ref{stcex} is not ad-hoc; a major advantage to using Spectral Tetris is the sparsity of the synthesis matrix which it constructs. It has been shown in \cite{CHKK}, that tight Spectral Tetris frames are optimally sparse in the sense that given $N \geq 2M$, the synthesis matrix of the $N$-element unit norm, tight Spectral Tetris frame for $\mathcal{H}_M$ is sparsest among all synthesis matrices of $N$-element unit norm, tight frames for $\mathcal{H}_M$. Next we present this sparsity result from \cite{CHKK}; but first we give a few necessary definitions and results.

\begin{definition}
Let $N\geq M>0$.

\begin{itemize}

\item  Let the real values $\lambda_1,\dots,\lambda_M \geq 2$ satisfy $\sum_{m=1}^M\lambda_m=N$. Then the class of unit norm frames $\{f_n\}_{n=1}^N$ in $\mathcal{H}_M$ whose frame operator has eigenvalues $\lambda_1,\dots,\lambda_M$ will be denoted by {\em $\mathcal{F}\left(N,\{\lambda_m\}_{m=1}^M\right)$}.

\item The $N$-element Spectral Tetris frame with eigenvalues $\lambda_1,\dots,\lambda_M\geq 2$ will be denoted by {\em $STF\left(N;\lambda_1,\dots,\lambda_M\right)$.}
\end{itemize}
\end{definition}

\begin{definition}
A finite sequence of real values $\lambda_1,\dots, \lambda_M$ is {\em ordered blockwise}, if for any permutation $\pi$ of $\{1,\dots, M\}$ the set of partial sums $\{\sum_{m=1}^s\lambda_m: s=1,\dots M\}$ contains at least as many integers as the set $\{\sum_{m=1}^s\lambda_{\pi\left(m\right)} :s=1,\dots,M\}$. The {\em maximal block number} of a finite sequence of real values $\lambda_1,\dots, \lambda_M$, denoted by $\mu\left(\lambda_1,\dots,\lambda_M\right)$, is the number of integers in $\{\sum_{m=1}^s\lambda_{\sigma \left(m\right)}:s=1,\dots,M\}$, where $\sigma$ is a permutation of $\{1,\dots, M\}$ such that $\lambda_{\sigma\left(1\right)},\dots,\lambda_{\sigma\left(M\right)}$ is ordered blockwise.
\end{definition}

Now we will present a useful theorem which provides a sparsity bound for any frame in $\mathcal{F}\left(N,\{\lambda_m\}_{m=1}^M\right)$.

\begin{lemma}\label{thm4.4g}\cite{CHKK}
Let $N\geq M >0$ and let the real values $\lambda_1,\dots,\lambda_M \geq 2$ satisfy $\sum_{m=1}^M\lambda_m=N$. Then any frame in $\mathcal{F}\left(N,\{\lambda_m\}_{m=1}^M\right)$ has sparsity at least $N+2\left(M-\mu\left(\lambda_1,\dots,\lambda_M\right)\right)$ with respect to any orthonormal basis of $\mathcal{H}_M$.
\end{lemma}

It is important to note that optimally sparse UNTFs in $\mathcal{F}\left(N,\{\lambda_m\}_{m=1}^M\right)$ are not uniquely determined. We will now give an example of two different UNTFs in $\mathcal{F}\left(9,\{\frac{9}{4}\}_{i=1}^9\right)$ with $M=4$ and $N=9$ which both achieve the optimal sparsity $9 + 2\left(4- 1\right) = 15$:
\[\left[\begin{array}{ccccccccc}
1&1&\sqrt{\frac{1}{8}}&\sqrt{\frac{1}{8}}&0&0&0&0&0\\
0&0&\sqrt{\frac{7}{8}}&-\sqrt{\frac{7}{8}}&\sqrt{\frac{1}{4}}&\sqrt{\frac{1}{4}}&0&0&0\\
0&0&0&0&\sqrt{\frac{3}{4}}&-\sqrt{\frac{3}{4}}&\sqrt{\frac{3}{8}}&\sqrt{\frac{3}{8}}&0\\
0&0&0&0&0&0&\sqrt{\frac{5}{8}}&-\sqrt{\frac{5}{8}}&1
\end{array}\right] \]

and

\[\left[\begin{array}{ccccccccc}
1&\sqrt{\frac{5}{8}}&\sqrt{\frac{5}{8}}&0&0&0&0&0&0\\
0&\sqrt{\frac{3}{8}}&-\sqrt{\frac{3}{8}}&\sqrt{\frac{3}{8}}&\sqrt{\frac{3}{8}}&\sqrt{\frac{3}{8}}&\sqrt{\frac{3}{8}}&0&0\\
0&0&0&\sqrt{\frac{5}{8}}&-\sqrt{\frac{5}{8}}&0&0&1&0\\
0&0&0&0&0&\sqrt{\frac{5}{8}}&-\sqrt{\frac{5}{8}}&0&1
\end{array}\right].\]

Thus unit norm, tight Spectral Tetris frames are not the only optimally sparse UNTFs. However, Spectral Tetris provides an easy algorithm for the construction of such frames. Now we present the precise statement which proves that Spectral Tetris constructs optimally sparse UNTFs in $\mathcal{F}\left(N,\{\lambda_m\}_{m=1}^M\right)$.

\begin{theorem}\cite{CHKK}\label{thmsparse}
Let $N\geq M >0$, then the Spectral Tetris UNTF $\{f_n\}_{n=1}^N$ with real eigenvalues $\lambda_1,\dots, \lambda_M\geq 2$ ordered blockwise satisfying $\sum_{m=1}^M\lambda_m=N$ is optimally sparse in $\mathcal{F}\left(N,\{\lambda_m\}_{m=1}^M\right)$ with respect to the standard unit vector basis. That is, this frame is $N+2\left(M-\mu\left(\lambda_1,\dots,\lambda_M\right)\right)$-sparse with respect to the standard unit vector basis. 
\end{theorem}

\begin{proof}\cite{CHKK}
Let $\{f_n\}_{n=1}^N$ be a unit norm, tight Spectral Tetris frame with eigenvalues $\lambda_1,\dots,\lambda_M \geq 2$. We will first show that its synthesis matrix has block decomposition of order $\mu:=\mu\left(\lambda_1,\dots,\lambda_M\right)$. For this, let $k_0 = 0$, and let $k_1,\dots,k_{\mu} \in \mathbb{N}$ be chosen such that $m_i
:= \sum_{m=1}^{k_i} \lambda_m$ is an integer for every $i = 1,\dots,\mu$. Moreover, let $m_0 = 0$. Further, note that $k_{\mu}= M$ and $m_{\mu} = N$, since $\sum_{m=1}^M\lambda_m$ is an integer by hypothesis. The steps of Spectral Tetris for computing STF$\left(m_1;\lambda_1,\dots,\lambda_{k_1}\right)$ and STF$\left(N;\lambda_1,\dots,\lambda_M\right)$ coincide until we reach the entry in the $k_1^{th}$ row and $m_1^{th}$ column when computing STF$\left(N;\lambda_1,\dots,\lambda_M\right)$. Therefore, the first $k_1$ entries of the first $m_1$ vectors of both constructions coincide. Continuing the computation of STF$\left(N,\lambda_1,\dots,\lambda_M\right)$ will set the remaining entries of the first $m_1$ vectors and also the first $k_1$ entries
of the remaining vectors to zero. Thus, any of the first $k_1$ vectors have disjoint support from any of the vectors constructed later on. Repeating this argument for $k_2$ until $k_{\mu}$, we obtain that the synthesis matrix has a block decomposition of order $\mu$; the corresponding partition of the frame vectors being
\[\bigcup_{i=1}^{\mu}\{f_{m_{i-1}+1},\dots,f_{m_i}\}.\]

To compute the number of non-zero entries in the synthesis matrix generated by Spectral Tetris, we let $i \in \{1,\dots,\mu\}$ be arbitrarily fixed and compute the number of non-zero entries of the vectors $f_{m_{i-1}+1},\dots,f_{m_i}$. Spectral Tetris ensures that each of the rows $k_{i-1}+1$ up to $k_i-1$ intersects the support of the subsequent row on a set of size 2, since in these rows Spectral tetris will always produce a $2 \times 2$ submatix $A\left(x\right)$ for some $0<x\leq 2$. Thus, there exist $2\left(k_i-k_{i-1}-1\right)$ frame vectors with two non-zero entries. The remaining $\left(m_i-m_{i-1}\right)-2\left(k_i-k_{i-1}-1\right)$ frame vectors will have only one entry, yielding a total number of $\left(m_i-m_{i-1}\right)+2\left(k_i-k_{i-1}-1\right)$ non-zero entries in the vectors $f_{m_{i-1}+1},\dots,f_{m_i}$.

Summarizing, the total number of non-zero entries in the frame vectors of $\{f_n\}_{n=1}^N$ is 
\[\sum_{i=1}^{\mu}\left(m_i-m_{i-1}\right)+2\left(k_i-k_{i-1}-1\right)=\]
\[\left(\sum_{i=1}^{\mu}\left(m_i-m_{i-1}\right)\right)+2\left(k_{\mu}-\left(\sum_{i=1}^{\mu}1\right)\right)=N+2\left(M-\mu\right),\]

which by Lemma \ref{thm4.4g} is the maximally achievable sparsity. 
\end{proof}

From Theorem \ref{thmsparse}, it would appear that Spectral Tetris constructs frames which are {\em only} optimally sparse with respect to the standard unit vector basis. However, if sparsity with respect to a different orthonormal basis is required, then Spectral Tetris can be modified by constructing the frame vectors with respect to this orthonormal basis instead. Moreover, this modified Spectral tetris algorithm constructs UNTFs which are optimally sparse with respect to this new orthonormal basis. Also note that this sparsity is dependent on the ordering of the given sequence of eigenvalues for which the Spectral Tetris construction is performed, which we will see in upcoming sections.

Spectral Tetris not only provides optimally sparse UNTFs, it also yields orthogonality between numerous pairs of frame vectors due to their disjoint support. This can be seen in Example \ref{stcex} where columns $t_{i,j}$ and $t_{i,j'}$ are orthogonal whenever $|j'-j| \geq 5$. More generally, any unit norm, tight {\it Spectral Tetris} frame, $\{f_n\}_{n=1}^N$ satisfies the orthogonality condition $\langle f_n, f_{n'}\rangle=0$ whenever $|n'-n| \geq \left\lfloor \frac{N}{M} \right\rfloor + 3$. This is explicitly stated in the following theorem:

\begin{theorem} \cite{CFMWZ09}
For any $M,N \in \mathbb{N}$ such that $N \geq 2M$, there exists a unit norm, tight frame $\{f_n\}_{n=1}^N $ for $\mathcal{H}_M$ with the property that $\langle f_n, f_{n'} \rangle = 0 $ whenever $|n'-n| \geq \left\lfloor \frac{N}{M} \right\rfloor+3$.
\end{theorem}

We have seen that Spectral Tetris provides an easy to use construction method for unit norm, tight frames which are extremely sparse and possess nice orthogonality properties. However, Spectral Tetris requires the frame to have at least twice as may vectors as the dimension, which is illustrated in the following example:

\begin{example}
We try to use Spectral Tetris to construct a UNTF with four vectors in $\mathbb{R}^3$, so the square sums of the rows of our matrix will be $\frac{4}{3}$. First we put a one in position $t_{1,1}$ and zeros in positions $t_{2,1}$ and $t_{3,1}$. Next, we need our building block $A\left(x\right)$ for positions $t_{1,2}, t_{1,3},t_{2,2}$ and $t_{2,3}$ to get:

$$T^*=\left[\begin{array}{cccc}
1&\sqrt{\frac{1}{6}}&\sqrt{\frac{1}{6}}&\cdot\\
0&\sqrt{\frac{5}{6}}&-\sqrt{\frac{5}{6}}&\cdot\\
0&0&0&\cdot
\end{array} \right].$$

But the square sum of the entries in row two is $\frac{5}{3}$, exceeding the required eigenvalue of $\frac{4}{3}$. Thus Spectral Tetris cannot construct such a UNTF.
\end{example}

Requiring the frame to have redundancy of at least two is a small constraint, which can easily be remedied. One way around this constraint is that we can acquire a unit norm, tight frame with less vectors, $N< 2M$, by constructing a corresponding unit-norm, tight Spectral Tetris frame, satisfying $N\geq 2M$, whose Naimark complement is the unit-norm tight frame we want. Hence Spectral Tetris ultimately constructs any UNTF with $N \geq M$ so long as the majorization condition is satisfied. 


\section{Spectral Tetris Constructions for Unit Norm Tight Frames with Redundancy Less than 2}

Although the use of the Naimark complement completely solves the construction problem for UNTFs via Spectral Tetris, it would be nice to be able to construct a UNTF with $N<2M$ explicitly using Spectral Tetris and cut out the additional Naimark complement step. Because of this, Spectral Tetris was adapted to construct such frames. We will see that in certain cases, Spectral Tetris can construct a UNTF with $N$ vectors in $M$ dimensions when $M \leq N <2M$. However, in general, when $M < N < 2M$ we need to adapt the $2 \times 2$ matrix $A\left(x\right)$ at (\ref{eqn:a}) used in Spectral Tetris and instead use larger submatrices. In particular, Spectral Tetris can construct a UNTF with redundancy $\frac{N}{M} \geq \frac{3}{2}$ through the use of a $3 \times 3$ submatrix, as we now have two diagonal entries over which to spread at most three units of spectral weight. We will see that the blocks themselves are obtained by scaling the rows of a $3 \times 3$ discrete Fourier transform matrix. More generally, UNTFs with redundancy greater that $\frac{j}{j-1}$ can be constructed using $J \times J$ submatrices. However, through the use of these larger submatrices, we lose some sparsity within the frame, which inevitably reduces the orthogonality between the frame vectors. 

There are some instances when $M < N <2M$ for which the original Spectral Tetris construction method will work to construct a UNTF and we characterize this completely in the following theorem.

\begin{theorem} \label{thm4.1}\cite{CHKWA}
For $M<N<2M$ and $\lambda=\frac{N}{M}$ the following are equivalent:
\begin{enumerate}
\item The Spectral Tetris construction will successfully produce a unit norm tight frame $\{ f_n \}_{n=1}^N$ for $\mathcal{H}_M$.

\item For all $1 \leq k \leq M-1$, if $k\lambda$ is not an integer, then we have $\left\lfloor k \lambda \right\rfloor \leq \left(k+1 \right) \lambda-2$, where $\left\lfloor x \right\rfloor$ is the greatest integer less than or equal to $x$.
\end{enumerate}
\end{theorem}

The condition in Theorem \ref{thm4.1} completely characterizes when Spectral Tetris will work to construct UNTFs with redundancy less than 2. Combine this with the fact that Spectral Tetris is known to work for redundancy greater than or equal to 2, and we see that we have completely classified when Spectral Tetris is able to construct UNTFs. We provide an example to illustrate the conditions in Theorem \ref{thm4.1}.

\begin{example}\label{evex}
In $\mathcal{H}_4$, construct a 6-element UNTF. Hence our tight frame bound/eigenvalue is $\lambda =\frac{6}{4}=\frac{3}{2} < 2.$ Next we will check if condition (2) holds: For all $1 \leq k \leq 3$,
\begin {itemize}
\item $1\left(\frac{3}{2}\right)=\frac{3}{2}$ is not at integer and $\left\lfloor 1\left(\frac{3}{2}\right) \right\rfloor=1 \leq 1=\left(1+1\right)\frac{3}{2}-2$.
\item $2\left(\frac{3}{2}\right)=3$ is an integer.
\item $3\left(\frac{3}{2}\right)=\frac{9}{2}$ is not at integer and $\left\lfloor 1\left(\frac{9}{2}\right) \right\rfloor=4 \leq 4=\left(3+1\right)\frac{3}{2}-2$.
 \end{itemize}
Thus condition (2) holds and therefore Spectral Tetris will construct this frame. Moreover, the frame constructed by Spectral Tetris is 
$$\left[\begin{array}{cccccc}
1&\sqrt{\frac{1}{4}}&\sqrt{\frac{1}{4}}&0&0&0\\
0&\sqrt{\frac{3}{4}}&-\sqrt{\frac{3}{4}}&0&0&0\\
0&0&0&1&\sqrt{\frac{1}{4}}&\sqrt{\frac{1}{4}}\\
0&0&0&0&\sqrt{\frac{3}{4}}&-\sqrt{\frac{3}{4}}
\end{array}\right].$$
\end{example}

The condition in Theorem \ref{thm4.1} is completely determined by the value of the tight frame bound $\lambda$. Moreover, we have an equivalent classification for when Spectral Tetris can be used to construct UNTFs which relies solely on the tight frame bound/eigenvalue $\lambda$. 

\begin{theorem} \label{thm4.2}\cite{CHKWA}
Spectral Tetris can be performed to generate a unit norm, tight frame of $N$ vectors in $\mathcal{H}_M$ if and only if, when $\lambda$ is in reduced form, one of the following occur:
\begin{enumerate}
\item $\lambda := \frac{N}{M} \geq 2$ or
\item $\lambda$ is of the form $\lambda = \frac{2L-1}{L}$ for some positive integer $L$.
\end{enumerate}
\end{theorem}

\begin{remark}
The requirement that $\lambda=\frac{M}{N}$ is in reduced form is crucial to property (2) in Theorem \ref{thm4.2}. Also, if $M$ and $N$ are known to be relatively prime, then property (2) is equivalent to $M = 2N - 1$.
\end{remark}

\begin{example}\label{4ex}
In Example \ref{evex}  we constructed a 6-element frame in $\mathcal{H}_4$ yielding the eigenvalue $\lambda=\frac{6}{4}=\frac{3}{2}=\frac{2\left(2\right)-1}{2}$. Hence by Theorem \ref{thm4.2}, Spectral Tetris can construct such a frame. Moreover, Theorem \ref{thm4.2} guarantees that Spectral Tetris can construct an $N$-element UNTF in $\mathcal{H}_4$ for all $N \geq 8$ and $N=6$. Also, since $\frac{7}{4}= \frac{2\left(4\right)-1}{4}$ then Spectral Tetris can construct a 7-element UNTF in $\mathcal{H}_4$. However, since $\frac{5}{4} \neq \frac{2L-1}{L}$ for any positive integer $L$ then Spectral Tetris cannot construct a 5-element UNTF in $\mathcal{H}_4$. Thus, Theorem \ref{thm4.2} completely classifies when Spectral Tetris can be used to construct UNTFs. However, there are cases for which we know UNTFs exist which Spectral Tetris cannot construct. Indeed, even though Spectral Tetris cannot construct a 5-element UNTF in $\mathcal{H}_4$, we know one exists because the majorization conditions on the eigenvalues and norms are satisfied in this case. 
\end{example}

It is clear that there exist UNTFs with redundancy less than two for which the conditions of Theorem \ref{thm4.2} are not satisfied. Next, we will explicitly see how Spectral Tetris can be adapted to construct such UNTFs. As mentioned earlier, in order to use Spectral Tetris to construct such a frame we will need larger building blocks than $A\left(x\right)$ at (\ref{eqn:a}). In particular, we will utilize discrete Fourier transform (DFT) matrices.

\begin{definition}
Given $M \in \mathbb{N}$, let $\omega = exp\left(\frac{2\pi k}{M}\right)$ be a primitive $M$-th root of unity. The (non-normalized) discrete Fourier transform (DFT) matrix in $\mathcal{H}_{M \times M}$ is defined by $F_M = \left( \omega^{ij}\right)_{i,j=0}^{M-1}$.
\end{definition}

\begin{remark}
DFT matrices possess the following nice properties:
\begin{enumerate}
\item The rows are orthogonal.
\item The columns are orthogonal.
\item All entries have the same modulus.
\end{enumerate}
\end{remark} 

Because of these nice properties of DFT matrices, we will use them as our building blocks; but in order to get the correct row norm and unit norm columns, we will need to alter the rows of the DFT matrix by multiplying by appropriate constants. Note that this will not affect the pairwise orthogonality of the rows.

\begin{example} \label{ex3.2}\label{dftex}
We will construct a 5-element unit norm tight frame in $\mathcal{H}_4$. Recall, in Example \ref{4ex} we showed that such a frame exists but the conventional Spectral Tetris method cannot construct this frame. 

To help illustrate why the original Spectral Tetris method fails here, we will attempt to start by adding $e_1$ in the first column. The next step would be to construct a $2\times 2$ matrix $A\left(x\right)$ where $x=\frac{3}{2}-1=\frac{1}{2}$. We get the following matrix:
$$\left[\begin{array}{ccccc}
1&\sqrt{\frac{1}{8}}&\sqrt{\frac{1}{8}}&0&0\\
0&\sqrt{\frac{7}{8}}&-\sqrt{\frac{7}{8}}&0&0\\
0&0&0&\cdot&\cdot\\
0&0&0&\cdot&\cdot
\end{array}\right].$$

Notice that row two square sums to $\frac{7}{4} > \frac{5}{4}$, so we already have too much weight in row 2 and thus this construction cannot work.

We will instead use altered DFT submatrices to construct a 5-element UNTF in $\mathcal{H}_4$. We will use the notation $\omega_M = exp\left(\frac{2\pi k}{M}\right)$. 

We will start by filling the desired $4 \times 5$ synthesis matrix with an altered $2 \times 2$ DFT matrix in the upper left corner. (Note that we could also use a standard $2 \times 2$ matrix $A\left(x\right)$ here.) The alteration we make is to multiply the entries of the first row by $\sqrt{\frac{5}{8}}$ in order to make the first row have the desired norm $\sqrt{\frac{5}{4}}$. 

In order to get unit norm columns, we need to multiply the second row of the $2 \times 2$ DFT matrix by $\sqrt{\frac{3}{8}}$. At this point we have constructed the first row and the first two columns of the desired synthesis matrix:
$$\left[\begin{array}{ccccc}
\sqrt{\frac{5}{8}}&\sqrt{\frac{5}{8}}&0&0&0\\
\sqrt{\frac{3}{8}}&\sqrt{\frac{3}{8}}\cdot \omega_2&\cdot&\cdot&\cdot\\
0&0&\cdot&\cdot&\cdot\\
0&0&\cdot&\cdot&\cdot
\end{array}\right].$$

Note that so far we have constructed a matrix whose first two rows are orthogonal regardless of how we finish filling in the second row. The second row at this point has norm $\sqrt{\frac{3}{4}}$, but we need to make it have norm $\sqrt{\frac{5}{4}}$. Hence we need to add a weight of $\sqrt{\frac{2}{4}}$ into row two. We cannot insert a $1 \times 1$ block of $\sqrt{\frac{2}{4}}$ because we would lose the orthogonality of the rows when making this column unit norm. Also we cannot insert an altered $2 \times 2$ DFT matrix in the same fashion as above because if we did we would have the following problem:

\begin{itemize}
\item To obtain the additional weight of $\sqrt{\frac{2}{4}}$ in row two, we would first need to multiply the first row of a $2 \times 2$ DFT by the factor $\sqrt{\frac{2}{8}}$.
\item Next, to obtain unit norm columns, this would force us to multiply the second row of the DFT by the factor $\sqrt{\frac{6}{8}}$.
\item Inserting this block into our synthesis matrix yields a norm of $\sqrt{\frac{12}{8}}=\sqrt{\frac{6}{4}}>\sqrt{\frac{5}{4}}$ our desired row norm. Thus we end up with too much weight in the fourth row of our synthesis matrix. 
\end{itemize}

Since we cannot insert another $2\times 2$ altered DFT, we next attempt to utilize an altered $3 \times 3$ DFT. To obtain the correct altered $3\times 3$ DFT we proceed as follows:
\begin{itemize}
\item First to obtain the additional weight of $\sqrt{\frac{2}{4}}$ in row two, we multiply the first row of a $3 \times 3$ DFT by the factor $\sqrt{\frac{2}{12}}=\sqrt{\frac{1}{6}}$.
\item Next, to obtain unit norm columns and row norms of $\sqrt{\frac{5}{4}}$ in the third and fourth row of the synthesis matrix, we multiply the second and third row of the $3 \times 3$ DFT by the factor $\sqrt{\frac{5}{12}}$.
\end{itemize}
This yields the desired $4 \times 5$ UNTF whose columns are normalized, rows are pairwise orthogonal and rows square sum to $\frac{5}{4}$.
$$\left[\begin{array}{ccccc}
\sqrt{\frac{5}{8}}&\sqrt{\frac{5}{8}}&0&0&0\\
\sqrt{\frac{3}{8}}&\sqrt{\frac{3}{8}}\cdot \omega_2&\sqrt{\frac{1}{6}}&\sqrt{\frac{1}{6}}&\sqrt{\frac{1}{6}}\\
0&0&\sqrt{\frac{5}{12}}&\sqrt{\frac{5}{12}} \cdot \omega_3&\sqrt{\frac{5}{12}}\cdot \omega_3^2\\
0&0&\sqrt{\frac{5}{12}}&\sqrt{\frac{5}{12}} \cdot \omega_3^2&\sqrt{\frac{5}{12}} \cdot \omega_3^4
\end{array}\right].$$

\end{example}

Notice that the Spectral Tetris construction in Example \ref{dftex} does follow a very similar format to our original Spectral Tetris construction; however in this latter example we are not just introducing vectors one or two at time, instead through the use of larger altered DFT submatrices we are introducing vectors in possibly larger increments corresponding to the size of these submatrices. In particular, in Example \ref{dftex} vectors are added in groups of two and three. Thus through these larger submatrices we are now able to construct UNTFs with redundancy less than two. Through all of this, we have seen the necessary and sufficient conditions for when Spectral Tetris can construct UNTFs with positive spectrum and we have seen an altered version of Spectral Tetris which allows us to construct even more UNTFs, all using an easily implemented construction process.


\section{Spectral Tetris for Non-Tight, Unit Norm Frames}\label{sfrsec}

Thus far we have seen how Spectral Tetris can construct UNTFs with a positive spectrum. Moreover, seeing that the construction of UNTFs via Spectral Tetris is now completely characterized, one might then ask the question: Can Spectral Tetris be used to construct non-tight, unit norm frames? In \cite{CCHKP10}, they answered this question positively and adapted Spectral Tetris to construct non-tight, unit norm frames with spectrum greater than or equal to two. They called this adaptation Sparse Unit Norm Frame Construction for Real Eigenvalues (SFR). Note that the spectrum of a finite frame is necessarily positive and real. Also note that since the frames SFR will construct are not necessarily tight then the rows of the frame need not square sum to the same constant. 

The SFR construction method also utilizes the $2 \times 2$ building block $A\left(x\right)$ at (\ref{eqn:a}), to help build a unit norm frame with prescribed spectrum one or two vectors at a time. In \cite{CCHKP10}, they provide sufficient conditions for when SFR can construct a unit norm frame. In particular, SFR can construct a unit norm frame $\{f_n\}_{n=1}^N$ in $\mathcal{H}_M$ with spectrum $\{\lambda_m\}_{m=1}^M \subseteq $ [2,$\infty$) if the following condition is satisfied:
\begin{itemize}
\item If $m_0$ is an integer in $\{1,\dots, M\}$, for which $\lambda_{m_0}$ is not an integer, then $\left\lfloor \lambda_{m_0} \right\rfloor \leq N-3$.
\end{itemize}
Note that this is only a sufficient condition for when SFR can construct a unit norm frame, in contrast to the necessary and sufficient conditions we previously had for Spectral Tetris constructing a UNTF. However, SFR provides the first general construction method for non-tight frames. When this condition is satisfied, we will see through a very easy to use algorithm, that SFR is very similar to the original Spectral Tetris construction method. Moreover, SFR also constructs extremely sparse frames and in fact the frames it constructs are always no more than 2-sparse.

In \cite{CCHKP10}, the authors provide an easily implementable algorithm, SFR, for constructing unit norm frames with prescribed spectrum, which is presented in Table 1.

\begin{table}[ht]
\centering
\framebox{
\begin{minipage}[ht]{4.8in}
\vspace*{0.3cm}
{\sc \underline{SFR: Sparse Unit Norm Frame Construction for Real Eigenvalues}}

\vspace*{0.4cm}

\noindent {\bf Parameters:}
\begin{itemize}
\item Dimension $M \in \mathbb{N}$.

\item Real eigenvalues $N \geq \lambda_1 \geq \dots\geq \lambda_M \geq 2$,  number of frame vectors $N$ satisfying $\sum_{m=1}^M\lambda_j=N\in\mathbb{N}$.
\end{itemize}

\noindent {\bf Algorithm:}
\begin{itemize}
\item Set $n=1$
\item For $m=1,\dots, M$ do 
		\begin{enumerate}
		\item Repeat
				\begin{enumerate}
				\item If $\lambda_m< 1$ then 
						\begin{enumerate}
						\item $f_n:=\sqrt{\frac{\lambda_m}{2}}\cdot e_m + \sqrt{1-\frac{\lambda_m}{2}}\cdot e_{m+1}$.
						\item $f_{n+1} := \sqrt{\frac{\lambda_m}{2}}\cdot e_m - \sqrt{1-\frac{\lambda_m}{2}}\cdot e_{m+1}$.
						\item $n:=n+2$.
						\item $\lambda_{m+1}:= \lambda_{m+1}-\left(2-\lambda_m\right)$.
						\item $\lambda_m:=0.$
						\end{enumerate}
				\item else
						\begin{enumerate}
						\item $f_n:=e_m$.
						\item $n:=n+1$.
						\item $\lambda_m:=\lambda_m-1$.
						\end{enumerate}
				\item{end}
				\end{enumerate}
		\item until $\lambda_m=0$.
		\end{enumerate}
\item end.
\end{itemize}

\noindent {\bf Output:}
\begin{itemize}
\item Unit norm frame $\{f_n\}_{n=1}^N$ with eigenvalues $\{\lambda_m\}_{m=1}^M$
\end{itemize}
\end{minipage}
}
\vspace*{0.2cm}
\caption{The SFR algorithm for constructing a 2-sparse, unit norm frame with a desired spectrum.}
\label{sfralg}
\end{table}

\begin{remark} 
For an explicit example of the SFR construction, see Remark \ref{sfrex}, which is based on Example \ref{sffrex} in Section \ref{sffrsection} of the present paper. 
\end{remark}

Although the SFR algorithm is the most helpful for constructing such frames, we will now provide the formal theorem which gives the sufficient conditions for when SFR can construct a unit norm frame with prescribed spectrum. 

\begin{theorem}\cite{CCHKP10}
Suppose that real values $\lambda_1 \geq \cdots \geq \lambda_M$ and $N\in \mathbb{N}$ satisfy $\sum_{j=1}^M\lambda_j=N$ (i.e. unit norm frame vectors) as well as the following conditions:
\begin{enumerate}
\item $\lambda_M \geq 2$,
\item If $m_0$ is an integer in $\{1,\dots, M\}$, for which $\lambda_{m_0}$ is not an integer, then $\left\lfloor \lambda_{m_0} \right\rfloor \leq N-3$.
\end{enumerate}

Then the eigenvalues of the frame operator of the frame $\{f_n\}_{n=1}^N$ constructed by SFR are $\{\lambda_m\}_{m=1}^M$ and the frame is 2-sparse.
\end{theorem}


\section{Generalized Spectral Tetris Frame Constructions}

By now, we have seen how Spectral Tetris has been used to construct UNTFs, originally with all eigenvalues greater than or equal to two and then adapted to construct UNTFs with all positive eigenvalues. Next, we saw how Spectral Tetris was then adapted to construct unit norm, non-tight frames with all eigenvalues greater than or equal to two. Logically, our next goal is to adapt Spectral Tetris to construct non-unit norm frames. In \cite{CHKWA}, they did just this and developed an easily implementable version of Spectral Tetris which constructs highly sparse frames with specified eigenvalues and specified vector norms. They also developed necessary and sufficient conditions on the eigenvalues and vector norms of a frame for when this construction will work. In doing this, they completely characterized the Spectral Tetris construction of a frame.

Recall our $2 \times 2$ building block $A\left(x\right)$, as defined at (\ref{eqn:a}). In this adaptation of Spectral Tetris, we will use a similar $2 \times 2$ building block and build our frame one or two vectors at a time. However, in order to allow for varied vector norms, we must modify property (1) of $A\left(x\right)$, which states that the columns of $A\left(x\right)$ square sum to 1, so that the columns of $A\left(x\right)$ can have varied norms; call these norms $a_1$ and $a_2$. Thus, the new $2 \times 2$ building blocks, which we denote $\hat{A}\left(x\right):=\hat{A}\left(x, a_1, a_2\right)$, will have the following properties:

\begin{enumerate}
\item The columns of $\hat{A}\left(x\right)$ have norms $a_1, a_2$ respectively.
\item The rows of $\hat{A}\left(x\right)$ are orthogonal. 
\item The square sum of the first row is $x$.
\end{enumerate}

These properties combined are equivalent to 
$$\hat{A}\left(x\right)\hat{A}^*\left(x\right)=\left[\begin{array}{cc}
x&0\\
0 & a_1^2+a_2^2-x
\end{array}\right].$$

A $2\times 2$ matrix which satisfies these conditions is:
$$\hat{A}\left(x\right):=\hat{A}\left(x,a_1,a_2\right)=\left[\begin{array}{cc}
\sqrt{\frac{x\left(a_1^2-y\right)}{x-y}}&\sqrt{\frac{x\left(x-a_1^2\right)}{x-y}}\\
\sqrt{\frac{y\left(x-a_1^2\right)}{x-y}}&-\sqrt{\frac{y\left(a_1^2-y\right)}{x-y}}\\
\end{array}\right],$$

where $y=a_1^2+a_2^2-x$. 

Similar to the original Spectral Tetris construction of a frame, this new adapted version of Spectral Tetris, called Prescribed Norms Spectral Tetris (PNSTC), builds frames one to two vectors at a time, where the $2\times 2$ matrix $\hat{A}\left(x\right)$ is used when two vectors are added. However, the existence of $\hat{A}\left(x\right)$ depends on $x, a_1$, and $a_2$. We will see in this new Spectral Tetris construction that the norms of the frame vectors correspond to the norms of the building block $\hat{A}\left(x\right)$, $a_1$ and $a_2$, and hence the existence of $\hat{A}\left(x\right)$ depends on $x$ and the norms of the frame vectors. In the following Lemma, we provide necessary and sufficient conditions for the existence of $\hat{A}\left(x\right)$.

\begin{lemma}\cite{CHKWA}\label{pnstclem}
A real matrix $\hat{A}\left(x\right):=\hat{A}\left( x, a_1, a_2 \right)$ satisfying $$\hat{A}\left(x\right)\hat{A}^*\left(x\right)=\left[\begin{array}{cc}
x&0\\
0 & a_1^2+a_2^2-x
\end{array}\right].$$

exists if and only if both of the following hold:
\begin{enumerate}
\item $a_1^2 + a_2^2 \geq x>0,$ and 
\item either $a_1^2, a_2^2 \geq x$ or $a_1^2, a_2^2 \leq x$.
\end{enumerate}
\end{lemma}

In order to construct frames with prescribed eigenvalues and prescribed norms, the original Spectral Tetris construction needs to be slightly modified. Not only does the $2 \times 2$ building block change to $\hat{A}\left(x\right)$, but in order to satisfy the conditions in Lemma \ref{pnstclem} there also needs to be some restrictions on the eigenvalue sequence and the vector norm sequence.

Recall that  a frame with vector norms $\{a_n\}_{n=1}^N$ and eigenvalues $\{\lambda_m\}_{m=1}^M$ exists if and only if after rearranging both sequences in non-increasing order and possibly adding zeros, $\{\lambda_m\}_{m=1}^M \succeq \{a_n^2\}_{n=1}^N$. This majorization condition is not sufficient in this scenario. In fact, a strengthening of majorization is required and is explicitly defined as follows:

\begin{definition} \label{def3.4} \cite{CHKWA}
We say two sequences $\{ a_n \}_{n=1}^N$ and $\{\lambda_m\}_{m=1}^M$ are \emph{Spectral Tetris ready} if $\sum_{n=1}^N a_n^2 = \sum_{m=1}^M \lambda_m$ and if there is a partition $0 \leq n_1 < \cdots < n_M=N$ of the set $\{0,1,\dots,N\}$ such that for all $k=1,2,\dots,M-1$:
\begin{enumerate}
\item $\sum_{n=1}^{n_k}a_n^2 \leq \sum_{m=1}^k\lambda_m < \sum_{n=1}^{n_k+1}a_n^2 \mbox{ and }$
\item $\mbox{ if } \sum_{n=1}^{n_k}a_n^2 < \sum_{m=1}^k\lambda_m, \mbox{ then } n_{k+1}-n_k \geq 2 \mbox{ and } a_{n_k+2}^2 \geq \sum_{m=1}^k \lambda_m-\sum_{n=1}^{n_k}a_n^2.$
\end{enumerate}
\end{definition}

Notice that there is no assumption on the ordering of the sequence of eigenvalues nor the sequence of vector norms in Definition \ref{def3.4}; hence we can permute the sequences, if necessary, to make them Spectral Tetris ready. It is important to note that some permutations of the sequences may be Spectral Tetris ready while other permutations may not. This is illustrated in the following example: 

\begin{example}
Given the eigenvalues $\{\lambda_m\}_{m=1}^3=\{8,6,4\}$ and the vector norms $\{a_n\}_{n=1}^4 =\{3,2,2,1\}.$ If we arrange the vectors norms as follows: $\{a_n\}_{n=1}^4 =\{2,1,3,2\}$, and take the partition $n_1=1, n_2=3$ and $n_3=4$, then the sequences $\{8,6,4\}$ and $\{2,1,3,2\}$ are Spectral Tetris ready. However, if we arrange the vector norms as follows: $\{a_n\}_{n=1}^4=\{2,2,3,1\}$ and leave the eigenvalues as $\{\lambda_m\}_{m=1}^3=\{8,6,4\}$, then no partition of the norms yields Spectral Tetris ready sequences. Also from this we see that the ordering of the sequences need not be monotone. 
\end{example}

Also, it may be the case that some sequences of eigenvalues and vector norms satisfy the majorization condition of Definition \ref{majorization} and hence there exists a corresponding frame; but no arrangement of the sequences is Spectral Tetris ready and as such PNSTC cannot construct such a frame. We illustrate this in the following example:

\begin{example}
We want to construct a 4-element tight frame in $\mathcal{H}_3$ with vector norms $\{a_n\}_{n=1}^4 =\{3,3,3,1\}.$ Hence our eigenvalue sequence is $\{\lambda_m\}_{m=1}^3=\{\frac{28}{3},\frac{28}{3},\frac{28}{3}\}$. Then $\{\lambda_m\}_{m=1}^3 \succeq \{a_n^2\}_{n=1}^4$ and hence such a frame exists. However, there is no arrangement of these eigenvalues and vector norms which is Spectral Tetris ready and hence PNSTC cannot construct such a frame.  
\end{example}

Due to this mild constraint on the sequences, PNSTC cannot construct construct all possible frames. In fact, the properties given in Definition \ref{def3.4} are exactly the necessary and sufficient conditions which allow PNSTC to construct a frame with prescribed norms and prescribed spectrum. This is stated explicitly in the following theorem.

\begin{theorem} \label{thm3.6}\cite{CHKWA}
Given $\{a_n\}_{n=1}^N \subseteq \left( 0, \infty \right)$ and $\{\lambda_m \}_{m=1}^M \subseteq \left(0,\infty\right)$, PNSTC can be used to construct
a frame $\{f_n\}_{n=1}^N$ for $\mathcal{H}_M$ such that $\|f_n\| = a_n$ for $n = 1, \dots,N$ and having eigenvalues $\{\lambda_m\}_{m=1}^M$ if and only if there exist a permutation which makes the sequences $\{a_n\}_{n=1}^N$ and $\{\lambda_m\}_{m=1}^M$ Spectral Tetris ready.
\end{theorem}

The algorithm presented in Table 2, from \cite{CHKWA}, is an adaptation of Spectral Tetris which constructs a frame with prescribed vector norms and prescribed eigenvalues, as long as they are Spectral Tetris ready. Similar to other forms of Spectral Tetris, this adapted version of Spectral Tetris, called Prescribed Norms Spectral Tetris (PNSTC), constructs a sparse frame one to two vectors at a time via an easily implementable algorithm.

\begin{table}[ht]
\centering
\framebox{
\begin{minipage}[ht]{4.8in}
\vspace*{0.3cm}
{\sc \underline{PNSTC: Prescribed Norms Spectral Tetris Construction}}

\vspace*{0.4cm}

\noindent {\bf Parameters:}
\begin{itemize}
\item Dimension $M \in \mathbb{N}$.

\item Number of frame elements $N \in \mathbb{N}$.

\item Eigenvalues $\{\lambda_m\}_{m=1}^M \subseteq \left(0,\infty \right)$ and norms of the frame vectors $\{a_n\}_{n=1}^N \subseteq \left(0,\infty \right)$ such that $\{\lambda_m\}_{m=1}^M$ and $\{a_n^2\}_{n=1}^N$ are Spectral Tetris ready.
\end{itemize}

\noindent {\bf Algorithm:}
\begin{itemize}
\item Set $n=1$
\item For $m=1,\dots, M$ do 
		\begin{enumerate}
		\item Repeat
				\begin{enumerate}
				\item If $\lambda_m \geq a_n^2$ then 
						\begin{enumerate}
						\item $f_n:=a_ne_m$.
						\item $\lambda_m := \lambda_m-a_n^2$.
						\item $n:=n+1$.
						\end{enumerate}
				\item else
						\begin{enumerate}
						\item If $2\lambda_m=a_n^2+a_{n+1}^2$, then
								\begin{enumerate}
								\item $f_n := \sqrt{\frac{\lambda_m}{2}}\cdot \left(e_m +e_{m+1}\right)$.
								\item $f_{n+1} := \sqrt{\frac{\lambda_m}{2}}\cdot \left(e_m -e_{m+1}\right)$.
								\end{enumerate}
						\item else
								\begin{enumerate}
								\item $y:=a_n^2+a_{n+1}^2 - \lambda_m$.
								\item $f_n := \sqrt{\frac{\lambda_m\left(a_n^2-y\right)}{\lambda_m-y}}\cdot e_m + \sqrt{\frac{y\left(\lambda_m- a_n^2\right)}{\lambda_m-y}} \cdot e_{m+1}$.
								\item $f_{n+1} := \sqrt{\frac{\lambda_m\left(\lambda_m-a_n^2\right)}{\lambda_m-y}}\cdot e_m - \sqrt{\frac{y\left(a_n^2-y\right)}{\lambda_m-y}} \cdot e_{m+1}$.
								\end{enumerate}
						\item end.
						\item $\lambda_{m+1}:=\lambda_{m+1}-\left(a_n^2+a_{n+1}^2 - \lambda_m\right).$
						\item $\lambda_m :=0$.
						\item $n:=n+2$.
						\end{enumerate}
				\item{end}
				\end{enumerate}
		\item until $\lambda_m=0$.
		\end{enumerate}
\item end.
\end{itemize}

\noindent {\bf Output:}
\begin{itemize}
\item Frame $\{f_n\}_{n=1}^N \subseteq \mathcal{H}_M$ with eigenvalues $\{\lambda_m\}_{m=1}^M$ and norms of the frame vectors $\{a_n\}_{n=1}^N$.
\end{itemize}

\end{minipage}
}
\vspace*{0.2cm}
\caption{The PNSTC algorithm for constructing a frame with prescribed spectrum and prescribed vector norms.}
\label{pnstcalg}
\end{table}

The PNSTC algorithm generalized Spectral Tetris to allow for any prescribed vector norms and any prescribed eigenvalues, so long as these sequences are Spectral Tetris ready. This algorithm is easily executable and the frames it constructs are extremely sparse, making them implementable in numerous applications. Next we present an example to further illustrate how PNSTC works.  

\begin{example}
Let us construct an 8-element frame in $\mathcal{H}_5$ with:
\[ \mbox{vector norms } \{a_n\}_{n=1}^8=\{4,1,2,\sqrt{3},1,\sqrt{2},3,2\}\]
\[\mbox{ and eigenvalues } \{\lambda_m\}_{m=1}^5= \{18,6,2,10,4\}.\]

First note that for PNSTC to work, we no longer need the assumption that we have at least twice as many vectors as the dimension, as this example will illustrate. Also note that, after rearranging both sequences in non-increasing order, our sequence of eigenvalues $\{\lambda_m\}_{m=1}^5= \{18,10, 6,4, 2\}$ majorizes the sequence of our square norms $\{a_n^2\}_{n=1}^8=\{4^2,3^2,2^2,2^2,\sqrt{3}^2,\sqrt{2}^2,1^2,1^2\}$. Therefore, a frame with these properties does exist. However, satisfying majorization does not guarantee that PNSTC can construct such a frame. We need to further guarantee that our sequences are Spectral Tetris ready, as defined in Definition \ref{def3.4}. If we arrange our sequences as we originally had them, i.e., $\{a_n\}_{n=1}^8=\{4,1,2,\sqrt{3},1,\sqrt{2},3,2\}$, $\{\lambda_m\}_{m=1}^5= \{ 18,6,2,10,4\}$, and let $n_1=2, n_2=4, n_3=5, n_4=7$, and $n_5=8$ then it is a straightforward check to see that our sequences are Spectral Tetris ready. 

Next, we will construct our frame using the PNSTC algorithm:
\begin{itemize}
\item Let $n=1$.
\begin{enumerate}
\item Let $m=1$
		\begin{enumerate}
		\item Is $\lambda_1=18 \geq 16 = 4^2 = a_1^2?$ Yes! Then we have the following:
				\begin{enumerate}
				\item $f_1:=a_1e_1=4e_1$.
				\item $\lambda_1:=\lambda_1-a_1^2=18-4^2=2$.
				\item $n:=n+1=1+1=2$.
				\item end.
				\end{enumerate}
		\item Does $\lambda_1=0$? No, $\lambda_1=2$ now, so repeat process. And now $\lambda_1=2$ and $n=2$.
		\item Is $\lambda_1=2 \geq 1^2=a_2^2$? Yes! Then we have the following:
				\begin{enumerate}
				\item $f_2:=a_2e_1=1e_1$.
				\item $\lambda_1:=\lambda_1-a_1^2=2-1^2=1$.
				\item $n:=n+1=2+1=3$.
				\item end.
				\end{enumerate}
		\item Does $\lambda_1=0$? No, $\lambda_1=1$ now, so repeat process. And now $\lambda_1=1$ and $n=3$.
		\item Is $\lambda_1=1 \geq 2^2=a_3^2$? No!
		\item Does $2\lambda_1=2= 7=2^2+\sqrt{3}^2=a_3^2+a_4^2$? No!
		\item Then let $y:=a_3^2+a_4^2-\lambda_1=2^2+\sqrt{3}^2-1=6$ and we have the following:
				\begin{enumerate}
				\item $f_3:=\sqrt{\frac{1\left(2^2-6\right)}{1-6}}\cdot e_1 + \sqrt{\frac{6\left(1-2^2\right)}{1-6}}\cdot e_2= \sqrt{\frac{2}{5}}\cdot e_1 + \sqrt{\frac{18}{5}}\cdot e_2$.
				\item $f_4:=\sqrt{\frac{1\left(1-2^2\right)}{1-6}}\cdot e_1 - \sqrt{\frac{6\left(2^2-6\right)}{1-6}}\cdot e_2= \sqrt{\frac{3}{5}}\cdot e_1 - \sqrt{\frac{12}{5}}\cdot e_2$.
				\item end.
				\end{enumerate}
		\item We also have the following:
				\begin{enumerate}
				\item $\lambda_2:=\lambda_2-\left(a_3^2+a_4^2-\lambda_1\right)=6-\left(4+3-1\right)=0$.
				\item $\lambda_1 := 0$.
				\item $n:=n+2=3+2=5$.
				\item end.
				\end{enumerate}
		\item Now $\lambda_1=0$ and we end this loop.
		\end{enumerate}
\end{enumerate}
\item Next $m=2$. (We still have $n=5$).
\item But $\lambda_2=0$ and we are done with this loop.
\item Next $m=3$. (We still have $n=5$).
		\begin{enumerate}
		\item Is $\lambda_3=2 \geq 1 = 1^2 = a_5^2?$ Yes! Then we have the following:
				\begin{enumerate}
				\item $f_5:=a_5e_3=1e_3$.
				\item $\lambda_3:=\lambda_3-a_5^2=2-1^2=1$.
				\item $n:=n+1=5+1=6$.
				\item end.
				\end{enumerate}
		\item Does $\lambda_3=0$? No, $\lambda_3=1$ now, so repeat process. And now $\lambda_3=1$ and $n=6$.
		\item Is $\lambda_3=1 \geq 2=a_6^2$? No!
		\item Does $2\lambda_3=2= 11=\sqrt{2}^2+3^2=a_6^2+a_7^2$? No!
		\item Then let $y:=a_6^2+a_7^2-\lambda_3=\sqrt{2}^2+3^2-1=10$ and we have the following:
				\begin{enumerate}
				\item $f_6:=\sqrt{\frac{1\left(2-10\right)}{1-10}}\cdot e_3 + \sqrt{\frac{10\left(1-2\right)}{1-10}}\cdot e_4= \sqrt{\frac{8}{9}}\cdot e_3 + \sqrt{\frac{10}{9}} \cdot e_4$.
				\item $f_7:=\sqrt{\frac{1\left(1-2\right)}{1-10}}\cdot e_3 - \sqrt{\frac{10\left(2-10\right)}{1-10}}\cdot e_4= \sqrt{\frac{1}{9}}\cdot e_3 - \sqrt{\frac{80}{9}}\cdot e_4$.
				\item end.
				\end{enumerate}
		\item We also have the following:
				\begin{enumerate}
				\item $\lambda_4:=\lambda_4-\left(a_6^2+a_7^2-\lambda_3\right)=10-\left(2+9-1\right)=0$.
				\item $\lambda_3 := 0$.
				\item $n:=n+2=6+2=8$.
				\item end.
				\end{enumerate}
		\item Now $\lambda_3=0$ and we end this loop.
		\end{enumerate}
\item Next $m=4$. (We still have $n=8$).
\item But $\lambda_4=0$ and we are done with this loop.
\item Next $m=5$. (We still have $n=8$).
		\begin{enumerate}
		\item Is $\lambda_5=4 \geq 4 = 2^2 = a_8^2?$ Yes! Then we have the following:
				\begin{enumerate}
				\item $f_8:=a_8e_5=2e_5$.
				\item $\lambda_5:=\lambda_5-a_8^2=4-4=0$.
				\item $n:=n+1=8+1=9$.
				\item end.
				\end{enumerate}
		\item Now $\lambda_5=0$ and we end this loop.
		\end{enumerate}
\item end.
\end{itemize}

{\bf Output:} PNSTC has just created an 8-element frame $\{ f_n\}_{n=1}^8$ of $\mathcal{H}_5$ with norms $\{a_n\}_{n=1}^8=\{4,1,2,\sqrt{3},1,\sqrt{2},3,2\}$ and eigenvalues $\{\lambda_m\}_{m=1}^5= \{ 18,6,2,10,4\}$. This frame is represented in the following matrix:

		$$\left[\begin{array}{cccccccc}
		4&1&\sqrt{\frac{2}{5}}&\sqrt{\frac{3}{5}}&0&0&0&0\\
		0&0&\sqrt{\frac{18}{5}}&-\sqrt{\frac{12}{5}}&0&0&0&0\\
		0&0&0&0&1&\sqrt{\frac{8}{9}}&\sqrt{\frac{1}{9}}&0\\
		0&0&0&0&0&\sqrt{\frac{10}{9}}&-\sqrt{\frac{80}{9}}&0\\
		0&0&0&0&0&0&0&2
		\end{array}\right].$$
		\end{example}

We have seen that PNSTC builds a frame one or two vectors at a time and utilizes a $2 \times 2$ building block $\hat{A}\left(x\right)$, which is a similar process to the original Spectral Tetris construction. However, before such a frame can be constructed, the eigenvalue sequence and vector norm sequence need to be Spectral Tetris ready, if possible. Finding a possible arrangement of these sequences can be a time-consuming and tedious task. To alleviate this long task, in \cite{CHKWA} they developed an easily-verified {\it sufficient} condition on the prescribed norms and eigenvalues under which PNSTC can be implemented. We state this condition in Proposition \ref{prop345}; but note that this condition is only a sufficient condition, whereas Spectral Tetris ready is a necessary and sufficient condition to apply PNSTC.

\begin{proposition}\label{prop345}\cite{CHKWA}
Let $\{a_n \}_{n=1}^N \subseteq \left( 0 ,\infty \right)$ and $\{ \lambda_m \}_{m=1}^M \subseteq \left( 0, \infty \right)$ be non-decreasing sequences such that $\sum_{n=1}^Na_n^2=\sum_{m=1}^M\lambda_m$ and \[a_{N-2L}^2+a_{N-2L-1}^2 \leq \lambda_{M-L}\] for $L=0,1,\dots,M-1$. Then  $\{ a_n\}_{n=1}^N$ and $\{ \lambda_m\}_{m=1}^M $ are Spectral Tetris ready, hence by Theorem \ref{thm3.6}, PNSTC can construct a frame $\{f_n\}_{n=1}^N$ for $\mathcal{H}_M$ with $\|f_n\|=a_n$ for $n=1,\dots,N$ and with eigenvalues $\{\lambda_m\}_{m=1}^M$. In particular, PNSTC can be performed if $a_N^2+a_{N-1}^2\leq \lambda_1$.
\end{proposition}

\begin{remark} In Proposition \ref{prop345}, the property $a_{N-2L}^2+a_{N-2L-1}^2 \leq \lambda_{M-L}$ together with $\sum_{n=1}^Na_n^2=\sum_{m=1}^M\lambda_m$ imply that $N \geq 2M$. Thus, this sufficient condition also requires redundancy of at least 2; whereas, the Spectral Tetris ready condition only required $N\geq M$.
\end{remark}

While PNSTC constructs frames with prescribed vector norms and prescribed spectrum, we can specialize PNSTC to construct tight frames and/or unit norm frames with the hopes that the conditions on the norm and eigenvalue sequences become easier to check. Moreover, if we wish to use PNSTC to construct a tight frame with prescribed norms, then the vector norms must satisfy the following easily verified sufficient condition. That is, in $M$ dimensions, with prescribed vector norms $\{a_n\}_{n=1}^N$ and tight frame bound $\lambda$, it suffices to check if $\lambda$ is greater than or equal to the sum of squares of the two largest vector norm values in $\{a_n\}_{n=1}^N$. Explicitly stated: 

\begin{theorem}\label{thm5.2}\cite{CHKWA}
Let $a_1 \geq a_2 \geq \cdots \geq a_N >0$ and $\lambda =\frac{1}{M}\sum_{n=1}^Na_n^2$. If $a_1^2+a_2^2 \leq \lambda$, then PNSTC constructs a $\lambda$-tight frame $\{f_n\}_{n=1}^N$ for $\mathcal{H}_M$ satisfying $\|f_n\|=a_n$ for all $n=1,2,\dots,N$.
\end{theorem}

\begin{remark} The condition in Theorem \ref{thm5.2} is an analog to the condition of the tight frame bound being at least 2 in the original Spectral Tetris construction, which ensured that Spectral Tetris works to produce unit norm, tight frames.
\end{remark}

Since the condition in Theorem \ref{thm5.2} is only a sufficient condition then we will see in the following example that there exist tight frames which fail the condition in Theorem \ref{thm5.2} but satisfy the Spectral Tetris ready condition. 

\begin{example}
Suppose we want to use PNSTC to construct a $6$-element tight frame in $\mathcal{H}_3$ with vector norms $\left(\sqrt{6},\sqrt{5},\sqrt{5},1,1,1\right)$. Then the tight frame bound will be $\lambda=\frac{19}{3}$. If we try to use Theorem \ref{thm5.2}, we see that $a_1^2+a_2^2=6+5=11 \not\leq \frac{19}{3}$ and so Theorem \ref{thm5.2} does not apply. However, if we rearrange the norms to $\left(\sqrt{6},1,\sqrt{5},1,1,\sqrt{5}\right)$ and take the partition $n_1=1,n_2=3$ and $n_3=6$ then these sequences are Spectral Tetris ready and hence PNSTC can construct such a frame. Therefore, although Theorem \ref{thm5.2} provides an easily-verifiable check on the vector norms, it does not provide a necessary condition for PNSTC to work.
\end{example}

However, by reformulating Definition \ref{def3.4} and Theorem \ref{thm3.6} to the case of tight frames with prescribed spectrum, we can get a necessary and sufficient condition for a sequence of norms to yield a tight frame via PNSTC.

\begin{corollary}\cite{CHKWA}\label{abc123}
A tight frame for $\mathcal{H}_M$ with prescribed norms $\{a_n\}_{n=1}^N$ and having eigenvalue $\lambda = \frac{1}{M}\sum_{n=1}^Na_n^2$ can be constructed via PNSTC if and only if there exists an ordering of $\{a_n^2\}_{n=1}^N$ for which there is a partition $0 \leq n_1 < \dots < n_M=N$ of $\{0,1,\dots,N\}$ such that for all $k=1,2,\dots, M-1$: 
\begin{enumerate}
\item $\sum_{n=1}^{n_k}a_n^2 \leq k\lambda < \sum_{n=1}^{n_k+1}a_n^2 \mbox{ for all } k=1,\dots, M-1, \mbox{ and }$
\item $\mbox{if } \sum_{n=1}^{n_k}a_n^2 < k\lambda, \mbox{ then } n_{k+1}-n_k \geq  \mbox{ and } a_{n_k+2}^2 \geq k\lambda - \sum_{n=1}^{n_k}a_n^2.$
\end{enumerate}
\end{corollary}

Although the conditions in Corollary \ref{abc123} require more work to check than the condition in Theorem \ref{thm5.2}, they are necessary and sufficient conditions and hence apply to a larger class of frames.

Another special case of PNSTC is the case of unit norm but not necessarily tight frames. Recall, in Section \ref{sfrsec} we saw a  sufficient condition for SFR to construct unit norm frames. This construction required the eigenvalues of a frame to be greater than or equal to two. In \cite{CHKWA}, the authors found necessary and sufficient conditions for SFR to construct unit norm frames with positive spectrum by reformulating Definition \ref{def3.4} and Theorem \ref{thm3.6}. 

\begin{corollary}\label{cor6.1}\cite{CHKWA}
Let $\sum_{m=1}^M\lambda_m=N$ where $N\in \mathbb{N}$ and $N\geq M$. Then SFR can be used to produce a unit norm frame for $\mathcal{H}_M$ with eigenvalues $\{\lambda_m\}_{m=1}^M \subseteq \left(0,\infty\right)$ if and only if there is some permutation of $\{\lambda_m\}_{m=1}^M$ such that there exists a partition $0\leq n_1 < \dots < n_M=N$ of $\{0,\dots,N\}$, such that for each $k=1,\dots,M-1$ we have
\begin{enumerate}
\item $n_k \leq \sum_{m=1}^k\lambda_m <n_k+1$ and
\item if $n_k < \sum_{m=1}^k\lambda_m$, then $n_{k+1}-n_k \geq 2$.
\end{enumerate}
\end{corollary}

The characterization in Corollary \ref{cor6.1} provides a strict limitation on the location of eigenvalues that can be strictly less than one, as we will see in the following Corollary. 

\begin{corollary}\cite{CHKWA}
If SFR can be used to produce a unit norm frame for $\mathcal{H}_M$ with eigenvalues $\left(\lambda_m\right)_{m=1}^M$, then $\lambda_k < 1$ is only possible if $k=1$ or if $n_{k-1}=\sum_{m=1}^{k-1}\lambda_m$.
\end{corollary}

The necessary and sufficient conditions of Corollary \ref{cor6.1} provide us with the information that SFR may not be able to construct a certain unit norm frame with eigenvalues $\{\lambda_m\}_{m=1}^M$ if these conditions are not met. However, if we loosen our conditions on the frame to be equal norm and not necessarily unit norm, then we will see that PNSTC will be able to construct an equal norm frame with those same eigenvalues $\{\lambda_m\}_{m=1}^M$.

\begin{theorem} \label{thm1234}\cite{CHKWA}
Let $\{\lambda_m\}_{m=1}^M \subseteq \left( 0, \infty \right)$ be non-increasing. Then PNSTC can construct an equal-norm frame for $\mathcal{H}_M$ with eigenvalues $\{\lambda_m\}_{m=1}^M$.
\end{theorem}

What originally started as a construction method specifically for unit norm tight Spectral Tetris frames, has now evolved into a complete characterization of Spectral Tetris frames. We have seen multiple adaptations of Spectral Tetris to construct UNTFs, unit norm frames and general frames. Finally, through PNSTC, we have a complete characterization for a Spectral Tetris construction of a frame with prescribed eigenvalues and prescribed vector norms. Also, we have seen that PNSTC can be specialized to include the previous Spectral Tetris cases to construct UNTFs, unit norm frames and tight frames. Aside from the fact that Spectral Tetris frames are easy to construct, their major advantage is the sparsity in the frames they construct. In fact, in all of the Spectral Tetris frame constructions, the frame vectors are at most 2-sparse. Even though Spectral Tetris cannot construct all frames, it can construct a large class of sparse frames and through the Spectral Tetris algorithms and classifications of eigenvalues and vector norms of a frame, researchers can now easily construct a large class of frames. Now that we have a complete characterization of Spectral Tetris frames, we wish to further Spectral Tetris to construct fusion frames. We will see that due to the sparsity of Spectral Tetris frames and the orthogonality of the frame vectors, we can generalize SFR and PNSTC to construct fusion frames.

\section{Spectral Tetris Fusion Frames Constructions}\label{ffsection}
Thus far we have seen how Spectral Tetris has developed and how it has been adapted to construct frames with desired properties. As mentioned earlier, it is not always the case in application that a problem can be solved with a single frame and may require two stage processing and the use of fusion frames. As such, we would like construction methods for fusion frames similar to that for frames so that researchers can easily construct fusion frames with desired properties. Seeing that fusion frames are a generalization of frames, then it is natural to think that SFR or PNSTC could be adapted to construct fusion frames. This is exactly correct and we will see that Spectral Tetris fusion frames stem nicely from Spectral Tetris frames. 

Although we have already discussed the relationship between frames and fusion frames in Section \ref{ffsec}, we reiterate the following relationship because it is essential in our Spectral Tetris fusion frame constructions.  

Consider the following fusion frame $\{\left(W_i,w_i\right)\}_{i=1}^D$ for $\mathcal{H}_M$ with frame operator $\widetilde{S}$. Let $\{f_{i,j}\}_{j=1}^{d_i}$ be an orthonormal basis for $W_i$ and $T$ the analysis operator for $\{f_{i,j}\}_{i=1,j=1}^{D,d_i}$, then we have the following equivalence:
\[\widetilde{S}x=\sum_{i=1}^Dw_i^2\left(P_i\left(x\right)\right)=\sum_{i=1}^D\sum_{j=1}^{d_i}w_i^2\langle x, f_{i,j}\rangle f_{i,j} =\]\[ \sum_{i=1}^D\sum_{j=1}^{d_i} \langle x, w_i f_{i,j}\rangle w_i f_{i,j}=TT^*x=Sx.\]

Thus we see that the fusion frame operator and the frame operator are equivalent in this scenario. Hence every fusion frame arises from a conventional frame partitioned into equal-norm, orthogonal sets. Because of this relationship and the fact that we have a complete characterization of Spectral Tetris frames, we would like to be able to construct a fusion frame via a Spectral Tetris frame with additional orthogonality requirements. This leads to the following terminology:

\begin{definition}\label{stffdef}
A frame constructed via the Spectral Tetris construction (SFR or PNSTC) is called a {\it Spectral Tetris frame}. A unit weighted fusion frame $\left(W_i\right)_{i=1}^D$ is called a \emph{Spectral Tetris fusion frame} if there is a partition of a Spectral Tetris frame $\{f_{i,j}\}_{i=1,j=1}^{D,d_i}$ such that $\{f_{i,j}\}_{j=1}^{d_i}$ is an orthonormal basis for $W_i$ for all $i = 1, \dots,D$.
\end{definition}

In Section \ref{sffrsection} and Section \ref{pnstcff}, the fusion frames we construct are unit weighted and thus in Definition \ref{stffdef}, we restrict ourselves to unit weighted fusion frames. However, in Section \ref{gffsec} we construct non-unit weighted fusion frames and this requires a more general definition, which we develop in that section. To construct our Spectral Tetris fusion frames, we will construct a Spectral Tetris frame and then group these frame vectors, so that each group of vectors is orthogonal and spans a subspace of the fusion frame.


\section{Spectral Tetris for 2-Sparse, Equidimensional, Unit-Weighted Fusion Frame Constructions}\label{sffrsection}

The first Spectral Tetris fusion frame construction occurred in \cite{CCHKP10}, where they adapted Spectral Tetris, more specifically SFR, to develop 2-sparse, equidimensional, unit-weighted fusion frames for any given fusion frame operator with eigenvalues greater than or equal to two. 

Before we discuss their construction algorithm, let's develop some background as to how it developed. Our goal is to determine the existence and construction of sparse fusion frames whose fusion frame operator possesses a desired spectrum. In particular, we would like to answer the questions:
\begin{enumerate}
\item Given a set of eigenvalues, does there exist a sparse fusion frame whose fusion frame operator possesses those eigenvalues?
\item If such a fusion frame exists how can it be constructed?
\end{enumerate}

In \cite{CCHKP10} they answer these questions for the case of unit norm, equidimensional, unit weighted fusion frames. 

Explicitly, in \cite{CCHKP10} they develop and analyze the following scenario:

Let $\lambda_1 \geq \dots \geq \lambda_M \geq 2$ be real values and $M \in \mathbb{N}$ satisfy the factorization 
\[\sum_{m=1}^M\lambda_m=kD\in\mathbb{N}.\]
Our goal is to construct a $2$-sparse fusion frame $\left(W_i\right)_{i=1}^D, W_i \subseteq \mathcal{H}_M$, such that:

(G1) dim$W_i = k$ for all $i = 1,\dots,D$ and

(G2) the associated fusion frame operator has $\{\lambda_m\}_{m=1}^M$ as its eigenvalues.

To construct such a fusion frame, in \cite{CCHKP10} they generalize the SFR algorithm and develop a new algorithm called Sparse Fusion Frame Construction for Real Eigenvalues (SFFR). The SFFR algorithm follows the same construction formula as the SFR algorithm; however, in the output stage of SFFR, the vectors $f_n$ are grouped in such a way so that the vectors assigned to each subspace form an orthonormal system. They also provide a sufficient condition for when the SFFR algorithm is able to construct a fusion frame which satisfies properties (G1) and (G2). This condition is very similar to the condition for SFR, and is as follows:
\begin{itemize}
\item If $m_0$ is the first integer in $\{1,\dots, M\}$ for which $\lambda_{m_0}$ is not an integer, then $\left\lfloor \lambda_{m_0} \right\rfloor \leq D-3$.
\end{itemize}

Hence if this condition is satisfied then SFFR will construct the desired fusion frame satisfying properties (G1) and (G2). Notice that the fusion frames constructable by SFFR must have all eigenvalues greater than or equal to two, just like in the SFR construction. In Table 3 we provide the SFFR algorithm from \cite{CCHKP10}, which constructs 2-sparse, equidimensional, unit-weighted fusion frames.

\begin{table}[ht]
\centering
\framebox{
\begin{minipage}[ht]{4.8in}
\vspace*{0.3cm}
{\sc \underline{SFFR: Sparse Fusion Frame Construction for Real Eigenvalues}}

\vspace*{0.4cm}

\noindent {\bf Parameters:}
\begin{itemize}
\item Dimension $M \in \mathbb{N}$.

\item Real eigenvalues $D \geq \lambda_1 \geq \dots\geq \lambda_M \geq 2$,  number of subspaces $D$, and dimension of subspaces $k$ satisfying $\sum_{m=1}^M\lambda_m=kD\in\mathbb{N}$.
\end{itemize}

\noindent {\bf Algorithm:}
\begin{itemize}
\item Set $j:=1$
\item For $m=1,\dots, M$ do 
		\begin{enumerate}
		\item Repeat
				\begin{enumerate}
				\item If $\lambda_m< 1$ then 
						\begin{enumerate}
						\item $f_j:=\sqrt{\frac{\lambda_m}{2}}\cdot e_m + \sqrt{1-\frac{\lambda_m}{2}}\cdot e_{m+1}$.
						\item $f_{j+1} := \sqrt{\frac{\lambda_m}{2}}\cdot e_m- \sqrt{1-\frac{\lambda_m}{2}}\cdot e_{m+1}$.
						\item $j:=j+2$.
						\item $\lambda_{m+1}:= \lambda_{m+1}-\left(2-\lambda_m\right)$.
						\item $\lambda_m:=0.$
						\end{enumerate}
				\item else
						\begin{enumerate}
						\item $f_j:=e_m$.
						\item $j:=j+1$.
						\item $\lambda_m:=\lambda_m-1$.
						\end{enumerate}
				\item{end}
				\end{enumerate}
		\item until $\lambda_m=0$.
		\end{enumerate}
\item end.
\end{itemize}

\noindent {\bf Output:}
\begin{itemize}
\item 2-sparse, equidimensional, unit weighted, fusion frame $\{W_i\}_{i=1}^D$ where $W_i:= \mbox{ span}\{f_{i+jD}:j=0,\dots, k-1\}$.
\end{itemize}

\end{minipage}
}
\vspace*{0.2cm}
\caption{The SFFR algorithm for constructing an 2 sparse, equidimensional, unit weighted, fusion frame with a desired frame operator.}
\label{sffralg}
\end{table}

To better illustrate the SFFR construction method, we will now provide an explicit example for constructing such a fusion frame. 

\begin{example}\label{sffrex}
We will construct a 2-sparse, equidimensional, unit-weighted fusion frame in $\mathcal{H}_3$ with 5 two-dimensional subspaces and spectrum $\{\lambda_m\}_{m=1}^3=\{ \frac{13}{3},\frac{10}{3},\frac{7}{3}\}$.

Note that $D=5\geq \frac{13}{3}\geq \frac{10}{3}\geq \frac{7}{3}\geq 2$ and $\sum_{m=1}^3\lambda_m=10=2\left(5\right)=kD\in\mathbb{N}$; hence the parameters of the algorithm are met. 

\begin{itemize}
\item Set $j=1$
\item For $m=1$ do 
		\begin{enumerate}
		\item $\lambda_1=\frac{13}{3} \geq 1$ then
						\begin{enumerate}
						\item $f_1:=e_1$.
						\item $j:=j+1=1+1=2$.
						\item $\lambda_1:=\lambda_1-1=\frac{13}{3}-1=\frac{10}{3}$.
						\end{enumerate}
		\item $\lambda_1=\frac{10}{3} \geq 1$ then
						\begin{enumerate}
						\item $f_2:=e_1$.
						\item $j:=2+1=3$.
						\item $\lambda_1:=\frac{10}{3}-1=\frac{7}{3}$.
						\end{enumerate}
		\item $\lambda_1=\frac{7}{3} \geq 1$ then
						\begin{enumerate}
						\item $f_3:=e_1$.
						\item $j:=3+1=4$.
						\item $\lambda_1:=\frac{7}{3}-1=\frac{4}{3}$.
						\end{enumerate}
			\item $\lambda_1=\frac{4}{3} \geq 1$ then
						\begin{enumerate}
						\item $f_4:=e_1$.
						\item $j:=4+1=5$.
						\item $\lambda_1:=\frac{4}{3}-1=\frac{1}{3}$.
						\end{enumerate}
			\item $\lambda_1=\frac{1}{3}< 1$ then 
						\begin{enumerate}
						\item $f_5:=\sqrt{\frac{\frac{1}{3}}{2}}\cdot e_1 + \sqrt{1-\frac{\frac{1}{3}}{2}}\cdot e_{2}=\sqrt{\frac{1}{6}}\cdot e_1 + \sqrt{\frac{5}{6}}\cdot e_{2}$.
						\item $f_6:=\sqrt{\frac{\frac{1}{3}}{2}}\cdot e_1 - \sqrt{1-\frac{\frac{1}{3}}{2}}\cdot e_{2}=\sqrt{\frac{1}{6}}\cdot e_1 - \sqrt{\frac{5}{6}}\cdot e_{2}$.
						\item $j:=5+2=7$.
						\item $\lambda_{m+1}=\lambda_2:= \lambda_{m+1}-\left(2-\lambda_m\right)=\frac{10}{3}-\left(2-\frac{1}{3}\right)=\frac{5}{3}$.
						\item $\lambda_1:=0.$
						\end{enumerate}
			\item end.
			\end{enumerate}
\item For $m=2$ (we have $\lambda_m=\lambda_2=\frac{5}{3}$ and $j=7$) do
		\begin{enumerate}
		\item $\lambda_2=\frac{5}{3} \geq 1$ then
						\begin{enumerate}
						\item $f_7:=e_2$.
						\item $j:=j+1=8$.
						\item $\lambda_2:=\lambda_2-1=\frac{2}{3}$.
						\end{enumerate}
		\item $\lambda_2=\frac{2}{3} < 1$ then
						\begin{enumerate}
						\item $f_8:=\sqrt{\frac{\frac{2}{3}}{2}}\cdot e_2 + \sqrt{1-\frac{\frac{2}{3}}{2}}\cdot e_{3}=\sqrt{\frac{1}{3}}\cdot e_2 + \sqrt{\frac{2}{3}}\cdot e_{3}$.
						\item $f_9:=\sqrt{\frac{\frac{2}{3}}{2}}\cdot e_2 - \sqrt{1-\frac{\frac{2}{3}}{2}}\cdot e_{3}=\sqrt{\frac{1}{3}}\cdot e_2 - \sqrt{\frac{2}{3}}\cdot e_{3}$.
						\item $j:=j+2=10$.
						\item $\lambda_{m+1}=\lambda_3:= \lambda_{3}-\left(2-\lambda_2\right)=\frac{7}{3}-\left(2-\frac{2}{3}\right)=1$.
						\item $\lambda_2:=0.$
						\end{enumerate}
		\item end.
		\end{enumerate}
\item For $m=3$ (we have $\lambda_m=\lambda_3=1$ and $j=10$) do
		\begin{enumerate}
		\item $\lambda_3=1 \geq 1$ then
						\begin{enumerate}
						\item $f_{10}:=e_3$.
						\item $j:=j+1=11$.
						\item $\lambda_3:=\lambda_3-1=0$.
						\end{enumerate}
		\item end.
		\end{enumerate}
\item end.
\end{itemize}

\noindent {\bf Output:}
\begin{itemize}
\item Define our two-dimensional subspaces $\{W_i\}_{i=1}^5$ as the following \[W_i:= \mbox{ span}\{f_{i+5j}:j=0,1\}.\] Explicitly, this yields:
\[W_1=\mbox{ span}\{f_1,f_6\}, W_2=\mbox{ span}\{f_2,f_7\}, W_3=\mbox{ span}\{f_3,f_8\},\] \[W_4=\mbox{ span}\{f_4,f_9\}, W_5=\mbox{ span}\{f_5,f_{10}\}.\]
It is straightforward to check that each of the subspaces $\{W_i\}_{i=1}^5$ is 2-dimensional and the spectrum of the fusion frame operator is $\{ \frac{13}{3},\frac{10}{3},\frac{7}{3}\}$. Therefore $\{W_i\}_{i=1}^5$ is a 2-sparse, equidimensional, unit weighted fusion frame with 5 two dimensional subspaces and spectrum $\{ \frac{13}{3},\frac{10}{3},\frac{7}{3}\}$. Hence the SFFR construction algorithm constructed the desired fusion frame.

\end{itemize}

\end{example}

\begin{remark}\label{sfrex}
Example \ref{sffrex} can be slightly simplified to also be an example of the SFR construction for a unit norm frame, see SFR Algorithm in Table 1. Explicitly in Example \ref{sffrex}, to adapt the SFFR algorithm construction to a construction for SFR our parameters and output would change to the following:

{\bf \emph{New SFR} Parameters:}
\begin{itemize}
\item Dimension $3 \in \mathbb{N}$.

\item Real eigenvalues $5 \geq \frac{13}{3}\geq \frac{10}{3} \geq \frac{7}{3} \geq 2$,  number of frame vectors $10$ satisfying $\sum_{m=1}^3\lambda_m=10=D\in\mathbb{N}$.
\end{itemize}

{\bf Algorithm:} The algorithm will be the exact same as in Example \ref{sffrex}.

{\bf \emph{New SFR} Output:}
\begin{itemize}
\item Unit norm frame $\{f_j\}_{j=1}^{10}$ with spectrum $\{\lambda_m\}_{m=1}^3=\{\frac{13}{3},\frac{10}{3},\frac{7}{3}\}$.
\end{itemize}
\end{remark}

We generalize the SFFR algorithm in the following Theorem.

\begin{theorem} \cite{CCHKP10}
Suppose the real values $D\geq \lambda_1\geq \dots \geq \lambda_M\geq 2, D \in \mathbb{N}$, and $k \in \mathbb{N}$ satisfy $\sum_{m=1}^M\lambda_m=kD\in\mathbb{N}$ as well as the following conditions:
\begin{enumerate}
\item $\lambda_M \geq 2$,
\item If $m_0$ is the first integer in $\{1,\dots, M\}$ for which $\lambda_{m_0}$ is not an integer, then $\left\lfloor \lambda_{m_0} \right\rfloor \leq D-3$.
\end{enumerate}

Then the fusion frame $\{W_i\}_{i=1}^D$ constructed by SFFR fulfills conditions (G1) and (G2) and the fusion frame is 2-sparse.
\end{theorem}

Through SFFR, we can construct a 2-sparse, equidimensional, unit weighted fusion frame with prescribed spectrum. Sometimes it is useful to extend such a fusion frame by adding more subspaces so that it becomes a tight fusion frame, since tight fusion frames possess nice reconstruction properties. The following Theorem provides sufficient conditions for when and what types of subsets can be added to a fusion frame in order to obtain a tight fusion frame.

\begin{theorem}\label{thm4.10}\cite{CCHKP10}
Let $\{W_i\}_{i=1}^D$ be a fusion frame for $\mathcal{H}_M$ with dim$W_i = k < M$ for all $i =1,\dots,D$, and let $\widetilde{S}$ be the associated fusion frame operator with eigenvalues $D\geq\lambda_1 \geq\dots \geq \lambda_M \geq 2$ and eigenvectors $\{e_m\}_{m=1}^M$. Further, let $A$ be the smallest positive integer, which satisfies the
following conditions:
\begin{enumerate}
\item $\lambda_1+2\leq A.$
\item $AM =kN_0$ for some $N_0 \in \mathbb{N}$.
\item $A \leq \lambda_M + N_0-\left(D+3\right)$.
\end{enumerate}

Then there exists a fusion frame $\{V_i\}_{i=1}^{N_0-D}$ for $\mathcal{H}_M$ with dim$V_i=k$ for all $i \in \{1,\dots,N_0-D\}$ so that $\{W_i\}_{i=1}^D \cup \{V_i\}_{i=1}^{N_0-D}$ is an $A$-tight fusion frame. 
\end{theorem}

The number of $k$-dimensional subspaces added in Theorem \ref{thm4.10} to extend a fusion frame to a tight fusion frame is in fact the smallest number that can be added in general.

We have seen that with little extra work, we can extend the SFR frame construction algorithm to construct a 2-sparse, unit weighted, equidimensional fusion frame with prescribed spectrum via the SFFR algorithm. Next we would like to be able to extend Spectral Tetris to construct a more general fusion frame, such as a non-equidimensional, unit weighted fusion frame with prescribed spectrum. 

\section{Spectral Tetris for Unit Weighted Fusion Frame Constructions}\label{pnstcff}

As we have seen, in \cite{CCHKP10} they adapted SFR to construct 2-sparse, equidimensional, unit weighted fusion frames with all eigenvalues greater than or equal to two. What if we wanted to construct a fusion frame with fewer restrictions? In \cite{CFHWZ} they generalized Spectral Tetris to construct unit-weighted fusion frames where the subspaces were not necessarily equidimensional and the eigenvalues need only to be positive. They also provide sufficient conditions for when this is possible, and provide necessary and sufficient conditions in the case of tight fusion frames with eigenvalues greater than or equal to two. 

To construct unit weighted fusion frames, we will first use PNSTC to construct a frame and then use this Spectral Tetris frame to obtain a {\it reference fusion frame}. This reference fusion frame is not our desired fusion frame, it is however a major step in the construction of our fusion frame. Before we present the Reference Fusion Frame (RFF) algorithm, we first need to define a few terms.

\begin{definition}
Given an $M \times N$ frame matrix $T^*=[f_1 \cdots f_N]$ representing an $N$-element frame in $\mathcal{H}_{M}$ against the eigenbasis of its frame operator, we have to following:
\begin{enumerate}
\item The {\it support size of a row} is the number of nonzero entries in that row. 
\item The {\it support of a frame vector $f_i$}, denoted supp$f_i$, is its nonzero entries.
\end{enumerate}
\end{definition}

\begin{definition}
Let $N \geq M$ be positive integers, and let $\{\lambda_m\}_{m=1}^M \subseteq \left(0, \infty \right)$ have the property that $\sum_{m=1}^M\lambda_m=N$. The fusion frame constructed by the algorithm RFF presented below in Table 4 is called the {\it reference fusion frame} for the eigenvalues $\left( \lambda_m\right)_{m=1}^M$.
\end{definition}
 
In Table 4 we present the Reference Fusion Frame Algorithm from \cite{CFHWZ}.

\begin{table}[ht]
\centering
\framebox{
\begin{minipage}[ht]{4.8in}
\vspace*{0.3cm}
{\sc \underline{RFF: Reference Fusion Frame Spectral Tetris Construction}}

\vspace*{0.4cm}

\noindent {\bf Parameters:}
\begin{itemize}
\item Dimension $M \in \mathbb{N}$.

\item Number of frame elements $N \in \mathbb{N}$.

\item Eigenvalues $\{\lambda_m\}_{m=1}^M \subseteq $ (0, $\infty$) such that $\sum_{m=1}^M\lambda_m=N$ (unit norm).
\end{itemize}

\noindent {\bf Algorithm:}
\begin{enumerate}
\item Use PNSTC for $\{\lambda_m\}_{m=1}^M$ with unit norm vectors to get a Spectral Tetris frame $F=\{f_n\}_{n=1}^N$.
\item $t:=$ maximal support size of the rows of F.
\item $S_i:=\emptyset$ for $i=1,\dots,t$.
\item $k=0$.
\item Repeat.
		\begin{enumerate}
		\item $k:=k+1$.
		\item $j:= \mbox{ min}\{1\leq r \leq t: \mbox{ supp} f_k \cap \mbox{ supp}f_s = \emptyset$ for all $f_s \in S_r\}$.
		\item $S_j:= S_j \cup \{f_k \}$.
		\end{enumerate}
\item until $k=N$.
\end{enumerate}	

\noindent {\bf Output:}
\begin{itemize}
\item Fusion frame $\left(V_i\right)_{i=1}^t$, where $V_i = \mbox{ span}\left(S_i\right)$ for $i=1,\dots, t.$
\end{itemize}
\end{minipage}
}
\vspace*{0.2cm}
\caption{The RFF algorithm for constructing the reference fusion frame.}
\label{rffalg}
\end{table}

It is important to note that in order to construct a reference fusion frame, the frame needs to be unit norm but not necessarily tight. 

\begin{example}\label{rffex}
Recall the unit norm tight frame with eigenvalues $\lambda=\frac{11}{4}$ constructed in Example \ref{stcex},
$$T^*=[f_1 f_2 \cdots f_{11}]=\left[\begin{array}{ccccccccccc}
1&1&\sqrt{\frac{3}{8}}&\sqrt{\frac{3}{8}}&0&0&0&0&0&0&0\\
0&0&\sqrt{\frac{5}{8}}&-\sqrt{\frac{5}{8}}&1&\sqrt{\frac{2}{8}}&\sqrt{\frac{2}{8}}&0&0&0&0\\
0&0&0&0&0&\sqrt{\frac{6}{8}}&-\sqrt{\frac{2}{8}}&1&\sqrt{\frac{7}{8}}&\sqrt{\frac{7}{8}}&0\\
0&0&0&0&0&0&0&0&\sqrt{\frac{7}{8}}&-\sqrt{\frac{7}{8}}&1
\end{array} \right].$$

It is a straight forward check of the RFF algorithm to see that the reference fusion frame given for frame $T^*$ is as follows:
\[V_1=\mbox{ span}\{f_1,f_5,f_8,f_{11}\}, V_2=\mbox{ span}\{f_2,f_6\},\] \[ V_3= \mbox{ span}\{f_3,f_9\}, V_4=\mbox{ span}\{f_4,f_{10}\}, V_5=\mbox{ span}\{f_7\}.\]
\end{example}

Note that different orderings of the eigenvalues of a frame will in general lead to different sequences of dimensions of the reference fusion frame, as the following example shows. In Example \ref{rffex}, we constructed the reference fusion frame from a unit norm tight frame and so we clearly do not have this issue. 

\begin{example}
We will construct a 10-element unit norm frame in $\mathcal{H}_3$ with eigenvalues $\{\lambda_m\}_{m=1}^3=\{\frac{13}{3},\frac{10}{3}, \frac{7}{3}\}$ using PNSTC/SFR and then construct its reference fusion frame. In Example \ref{sffrex} we already constructed the corresponding frame, which is as follows:
$$[f_1\cdots f_8]=\left[\begin{array}{cccccccccc}
1&1&1&1&\sqrt{\frac{1}{6}}&\sqrt{\frac{1}{6}}&0&0&0&0\\
0&0&0&0&\sqrt{\frac{5}{6}}&-\sqrt{\frac{5}{6}}&1&\sqrt{\frac{1}{3}}&\sqrt{\frac{1}{3}}&0\\
0&0&0&0&0&0&0&\sqrt{\frac{2}{3}}&-\sqrt{\frac{2}{3}}&1
\end{array}\right].$$

Thus the reference fusion frame constructed by RFF is as follows:
\[V_1=\mbox{ span}\{f_1,f_7,f_{10}\}, V_2=\mbox{ span}\{f_2,f_8\},\] \[ V_3= \mbox{ span}\{f_3,f_9\}, V_4=\mbox{ span}\{f_4\}, V_5=\mbox{ span}\{f_5\}, V_6=\mbox{ span}\{f_6\}.\]

However, if we reorder the same eigenvalues in the following way: $\{\lambda_m\}_{m=1}^3=\{\frac{7}{3},\frac{13}{3},\frac{10}{3}\}$, then PNSTC yields the following frame:

$$[g_1\cdots g_{10}]=\left[\begin{array}{cccccccccc}
1&1&\sqrt{\frac{1}{6}}&\sqrt{\frac{1}{6}}&0&0&0&0&0&0\\
0&0&\sqrt{\frac{5}{6}}&-\sqrt{\frac{5}{6}}&1&1&\sqrt{\frac{1}{3}}&\sqrt{\frac{1}{3}}&0&0\\
0&0&0&0&0&0&\sqrt{\frac{2}{3}}&-\sqrt{\frac{2}{3}}&1&1
\end{array}\right].$$

Thus the reference fusion frame which RFF constructs for this frame is:
\[V_1=\mbox{ span}\{g_1,g_5,g_9\}, V_2=\mbox{ span}\{g_2,g_6,g_{10}\},\] \[ V_3= \mbox{ span}\{g_3\}, V_4=\mbox{ span}\{g_4\}, V_5=\mbox{ span}\{g_7\}, V_6=\mbox{ span}\{g_8\}.\]

\end{example}

Hence different orderings can lead to different reference fusion frames, which as we will see, will alter the steps in our final fusion frame algorithm. Next we will use the reference fusion frame to help us construct our desired unit weighted fusion frame. The following Theorem \ref{thm4.4} provides sufficient conditions for when a Spectral Tetris fusion frame exists; but we first need the definition of a {\it chain}.

\begin{definition}
Let $S$ be a set of vectors in $\mathcal{H}_M$, and $s \in S$. We say that a subset $C \subseteq S$ is a {\it chain in $S$ starting at s}, if $s \in S$ and the support of any element in $S$ intersects the support of some element of $S$. We say that $C$ is a {\it maximal chain in $S$ starting at $s$} if $C$ is not a proper subset of any other chain in $S$ starting at $s$.
\end{definition}

\begin{theorem}\label{thm4.4}\cite{CFHWZ}
Let $N\geq M$ be positive integers, $\left(\lambda_m\right)_{m=1}^M \subseteq \left(0,\infty \right)$ and let $\left(d_i\right)_{i=1}^D \subseteq \mathbb{N}$ be a non-increasing sequence of dimensions such that $\sum_{m=1}^M\lambda_m=\sum_{i=1}^Dd_i=N$. Let $\left(V_i\right)_{i=1}^t$ be the references fusion frame for $\left(\lambda_m\right)_{m=1}^M$. If we have the majorization $\left(\mbox{dim }V_n\right)_{n=1}^t \succeq \left(d_i\right)_{i=1}^D$, then there exists a Spectral Tetris fusion frame $\left(W_i\right)_{i=1}^D$ for $\mathcal{H}_M$ with dim$W_i=d_i$ for $n=1,\dots,D$ and eigenvalues $\left(\lambda_m\right)_{m=1}^M$. 
\end{theorem}

Next, we provide the proof for Theorem \ref{thm4.4} from \cite{CFHWZ} because it is very constructive in nature and helps the reader to determine how the fusion frame subspaces are developed. Also, the proof leads nicely into the fusion frame construction algorithm. 

\begin{proof}
We show how to iteratively construct the desired fusion frame $\left(W_i\right)_{i=1}^D$. Let $t$ and $S_1,\dots,S_t$ be given by RFF for $\left(\lambda_m\right)_{m=1}^M$. Let $W_i^0=S_i$ for $i=1,\dots,t$. We add empty sets if necessary to obtain a collection $\left( W_i^0\right)_{i=1}^D$ of $D$ sets. If $\sum_{i=1}^D ||W_i^0|-d_i|=0$ then the sets $\left(W_i^0\right)_{i=1}^D$ span the desired fusion frame. Otherwise, starting from $\left(W_i^0\right)_{i=1}^D$, we will construct the spanning sets of the desired fusion frame. 

Let \[m=\mbox{ max}\{j \leq D: d_j \neq |W_j^0|\}.\]

Note that $\sum_{i=1}^m|W_i^0|=\sum_{i=1}^md_i$ by the choice of $m$, and $\sum_{i=1}^{m-1}|W_i^0|>\sum_{i=1}^{m-1}d_i$ by the majorization assumption. Therefore, $d_m > |W_m^0|$ and there exists \[k= \mbox{ max}\{j<m:|W_j^0|>d_j\}.\] 

Notice that $|W_m^0|<d_m\leq d_k<|W_k^0|$ implies $|W_m^0|+2\leq |W_k^0|$.

We now have to consider two cases:

{\bf Case 1:}

If there exists at least one element $w \in W_k^0$, which has disjoint support from every element in $W_m^0$, then pick one such $w \in W_k^0$ satisfying this property. Define $\left(W_i^1\right)_{i=1}^D$ by:

$$W_i^1=\left\{\begin{array}{cc}
W_k^0\setminus \{w\} & \mbox{ if } i =k,\\
W_m^0 \cup \{w\}& \mbox{ if } i=m,\\
W_i^0& \mbox{ else. }
\end{array}\right.$$

{\bf Case 2:}
If there is {\it no} such element $w \in W_k^0$ which has disjoint support from every element in $W_m^0$, then partition $W_k^0 \cup W_m^0$ into maximal chains, say $C_1,\dots, C_r$. Note that for each $i=1,\dots,r$, the cardinality of the sets $C_i \cap W_k^0$ and $C_i \cap W_m^0$ differ by at most one, since, given $v_k\in W_k^0$ and $v_m\in W_m^0$, we know that $v_k$ and $v_m$ either have disjoint support, or their support sets have intersection of size one. Since $|W_m^0|+2\leq |W_k^0|$ then there is a maximal chain $C_j$ that contains one element more from $W_k^0$ than from $W_m^0$. Define $\left(W_i^1\right)_{i=1}^D$ by:

$$W_i^1=\left\{\begin{array}{cc}
\left(W_k^0 \cup C_j\right) \setminus \left(W_k^0 \cap C_j\right) & \mbox{ if } i =k,\\
\left(W_m^0 \cup C_j\right) \setminus \left(W_m^0 \cap C_j\right) & \mbox{ if } i=m,\\
W_i^0& \mbox{ else. }
\end{array}\right.$$

In both of the above cases, we have defined $\left(W_i^1\right)_{i=1}^D$ such that \[\sum_{i=1}^D ||W_i^1|-d_i| < \sum_{i=1}^D||W_i^0|-d_i|.\]

Note that $\left(W_i^1\right)_{i=1}^D$ satisfies the majorization condition $\left(|W_i^1|\right)_{i=1}^D \succeq \left(d_n\right)_{n=1}^N$. Thus if the sets of $\left(W_i^1\right)_{i=1}^D$ do not span the desired fusion frame, we can repeat the above procedure with $\left(W_i^1\right)_{i=1}^D$ instead of $\left(W_i^0\right)_{i=1}^D$ and get $\left(W_i^2\right)_{i=1}^D$ such that $\sum_{i=1}^D||W_i^2|-d_i|<\sum_{i=1}^D||W_i^1|-d_i|$. Continuing in this fashion we will, say after repeating the process $l$ times, arrive at $\left(W_i^l\right)_{i=1}^D$ such that $\sum_{i=1}^D||W_i^l|-d_i|=0$, then the sets of $\left(W_i^l\right)_{i=1}^D$ span the desired fusion frame $\left(W_n\right)_{n=1}^D$.

\end{proof}

In Table 5 we provide an easily implementable algorithm to construct such a unit-weighted fusion frame as described in Theorem \ref{thm4.4}.

\begin{table}[ht]
\centering
\framebox{
\begin{minipage}[ht]{4.8in}
\vspace*{0.3cm}
{\sc \underline{UFF: Unit-Weighted Fusion Frame Spectral Tetris Construction}}

\vspace*{0.4cm}

\noindent {\bf Parameters:}
\begin{itemize}
\item Dimension $M \in \mathbb{N}$.

\item Number of frame elements $N \in \mathbb{N}$.

\item Eigenvalues $\left(\lambda_m\right)_{m=1}^M \subseteq $ (0, $\infty$) and dimensions $M> d_1\geq d_2 \geq \dots \geq d_D>0$ such that $\sum_{m=1}^M\lambda_m= \sum_{i=1}^Dd_i=N$.

\item Reference fusion frame $\left(V_i\right)_{i=1}^t$ for $\left(\lambda_m\right)_{m=1}^M$ such that $\left( \mbox{dim }V_i\right)_{i=1}^t \succeq \left(d_i \right)_{i=1}^D$.
\end{itemize}

\noindent {\bf Algorithm:}
\begin{enumerate}
\item Set $\ell :=0$
\item Set $W_i^{\ell} :=S_i$ for $0< i \leq t$ and $W_i^{\ell}:=\emptyset$ for $t< i \leq D$, do
\item Repeat
				\begin{enumerate}
				\item If $\sum_{i=1}^D||W_i^{\ell}|-d_i| \neq 0$
								\begin{enumerate}
								\item Set $m=\mbox{ max}\{j\leq D| d_j \neq |W_j^{\ell}|\}$
								\item Set $k=\mbox{ max}\{j<m| |W_j^{\ell}|>d_j\}$
								\item If $A= \{x \in W_k^{\ell}| \mbox{ supp}\left(x\right) \cap \mbox{ supp}\left(v\right) =\emptyset \mbox{ for all } v \in W_m^{\ell}\}\neq \emptyset$, then
												\begin{enumerate}
												\item Pick one $\hat{x} \in A$
												\item $W_k^{\ell +1} := W_k^{\ell} \setminus \{\hat{x}\}$
												\item $W_m^{\ell +1} := W_m^{\ell} \cup \{\hat{x}\}$
												\item $W_i^{\ell +1} :=W_i^{\ell} \mbox{ for all } i \neq k,m$
												\end{enumerate}
								\item else
												\begin{enumerate}
												\item Partition $W_k^{\ell} \cup W_m^{\ell}$ into maximal chains
												\item Pick one such maximal chain, $C_j$, which contains one more element from $W_k^{\ell}$ than from $W_m^{\ell}$
												\item $W_k^{\ell+1} := \left(W_k^{\ell} \cup C_j\right) \setminus \left( W_k^{\ell} \cap C_j \right)$
												\item $W_m^{\ell+1} := \left(W_m^{\ell} \cup C_j\right) \setminus \left( W_m^{\ell} \cap C_j \right)$
												\item $W_i^{\ell+1} := W_i^{\ell} \mbox{ for all } i \neq k,m$
												\end{enumerate}
								\item Set $\ell:=\ell+1$
								\end{enumerate}
				\item end.
				\end{enumerate}
\item Do until $\sum_{i=1}^D ||W_i^{\ell}|-d_i|=0$
\item end.
\end{enumerate}

\noindent {\bf Output:}
\begin{itemize}
\item The sets $\left(W_i^{\ell}\right)_{i=1}^D$ span the desired fusion frame $\left(W_i\right)_{i=1}^D$, where $W_i=\mbox{ span}\left(W_i^{\ell}\right)$ for all $i = 1,\dots, D$.
\end{itemize}
\end{minipage}
}
\vspace*{0.2cm}
\caption{The UFF algorithm for constructing a unit-weighted fusion frame.}
\label{uffalg}
\end{table}

We will now present an illustrative example to construct a unit weighted fusion frame using UFF. We will extend Example \ref{stcex} and Example \ref{rffex} to construct the corresponding unit-weighted fusion frame via UFF.

\begin{example}
We will construct a unit-weighted fusion frame in $\mathcal{H}_4$ with 11 frame elements, eigenvalues $\left(\frac{11}{4},\frac{11}{4},\frac{11}{4},\frac{11}{4}\right)$ and dimensions $ 3 \geq 3 \geq 2 \geq 1 \geq 1 \geq 1 $. 

Notice that $\sum_{m=1}^4\lambda_m=11= \sum_{i=1}^6d_i=N$.

The reference fusion frame this constructs is $\left(V_i\right)_{i=1}^5$ defined by the span of the sets: \[S_1=\{f_1,f_5,f_8,f_{11}\}, S_2=\{f_2,f_6\},\] \[ S_3= \{f_3,f_9\}, S_4=\{f_4,f_{10}\}, S_5=\{f_7\},\]
as seen in Example \ref{rffex}. Note that the majorization condition, $\left( \mbox{dim }V_i\right)_{i=1}^5 \succeq \left(d_i \right)_{i=1}^6$, is also satisfied.
\begin{itemize}
\item Set $\ell :=0$
\item Set $W_i^{0}:=\emptyset$ for $5< i \leq 6$. Hence we have the following sets:
				\[ W_1^0:=\{f_1,f_5,f_8,f_{11}\}; W_2^0:=\{f_2,f_6\}; W_3^0:=\{f_3,f_9\};\]
				\[ W_4^0:=\{f_4,f_{10}\}; W_5^0:=\{f_7\}; W_6^0:=\emptyset .\]
\item $\sum_{i=1}^6||W_i^{0}|-d_i|=4 \neq 0$
				\begin{enumerate}
								\item Set $m:=\mbox{ max}\{j\leq 6|d_j \neq |W_j^{0}|\}=6$
								\item Set $k:=\mbox{ max}\{j<6| |w_j^{0}|>d_j\}=4$
								\item Is $A= \{x \in W_4^{0}| \mbox{ supp}\left(x\right) \cap \mbox{supp}\left(v\right)=\emptyset \mbox{ for all } v \in W_6^{0}\}\neq\emptyset$? Yes! This is clear since $W_6^0=\emptyset$ then all of $W_4 \in A$. Then
												\begin{enumerate}
												\item Pick one $\hat{x} \in A$. We can pick $f_{10}$.
												\item Now we have the following new subspaces: 
															\[ W_1^1:=W_1^0=\{f_1,f_5,f_8,f_{11}\}; W_2^1:=W_2^0=\{f_2,f_6\};\] \[ W_3^1:W_3^0=\{f_3,f_9\}; W_4^1: W_4^{0} \setminus \{f_{10}\}=\{f_4\};\]\[ W_5^1:=W_5^0=\{f_7\}; W_6^1:= W_6^{0} \cup \{f_{10}\}=\{f_{10}\}.\]
												\end{enumerate}
								\item Set $\ell:=0+1=1$.
				\end{enumerate}
\item Repeat with $\ell=1$.	
\item $\sum_{i=1}^6||W_i^{1}|-d_i|=2 \neq 0$
				\begin{enumerate}
								\item Set $m:=\mbox{ max}\{j\leq 6|d_j \neq |W_j^{1}|\}=2$.
								\item Set $k:=\mbox{ max}\{j<2| |w_j^{1}|>d_j\}=1$.
								\item We have $A= \{x \in W_1^{1}| \mbox{ supp}\left(x\right) \cap \mbox{supp}\left(v\right)=\emptyset \mbox{ for all } v \in W_2^{1}\}=\{f_{11}\} \neq \emptyset$.
								\begin{enumerate}
												\item Pick one $\hat{x} \in A$. We can pick $f_{11}$.
												\item Now we have the following new subspaces: 
															\[ W_1^2:=W_1^1\setminus \{f_{11}\}=\{f_1,f_5,f_8\};\] \[W_2^2:=W_2^1 \cup \{f_{11}\}=\{f_2,f_6,f_{11}\};\] \[ W_3^2:W_3^1=\{f_3,f_9\}; W_4^2: W_4^{1}=\{f_4\};\]\[ W_5^2:=W_5^1=\{f_7\}; W_6^2:= W_6^{1}=\{f_{10}\}.\]
												\end{enumerate}
								
								\item Set $\ell:=1+1=2$
				\end{enumerate}
\item Repeat with $\ell=2$.
\item $\sum_{i=1}^6||W_i^{2}|-d_i|=0$
\item end.
\end{itemize}

{\bf Output:}
\begin{itemize}
\item The sets $\left(W_i^{2}\right)_{i=1}^6$ span the desired fusion frame $\left(W_i\right)_{i=1}^D$, where $W_i=\mbox{ span}\left(W_i^{2}\right)$ for all $i = 1,\dots, 6$.
\end{itemize}
\end{example}

The UFF algorithm and Theorem \ref{thm4.4} are useful and easily implementable when constructing unit weighted fusion frames with prescribed dimensions and prescribed positive spectrum. Note that these fusion frames are also 2-sparse like the fusion frames in the previous section were. However, Theorem \ref{thm4.4} only provides sufficient conditions for when UFF constructs a unit weighted fusion frame. However, if we consider the special case of unit weighted tight fusion frames then \cite{CFHWZ} provides necessary and sufficient conditions for when UFF can be applied. We will see that the majorization condition $\left(\mbox{dim }V_i\right)_{i=1}^t \succeq \left(d_i \right)_{i=1}^D$ in Theorem \ref{thm4.4} is also necessary for unit weighted tight fusion frames; however, we need the further requirement that $N\geq 2M$ for this to hold.

\begin{theorem}\cite{CFHWZ}\label{tightff}
Let $N \geq 2M$ be positive integers and $\{d_i\}_{i=1}^D \subseteq \mathbb{N}$ in non-increasing order such that $\sum_{i=1}^Dd_i=N$. Let $\left(V_i\right)_{i=1}^t$ be the reference fusion frame for $\{\lambda_m\}_{m=1}^M =\{\frac{N}{M},\cdots,\frac{N}{M}\}$. Then there exists a unit weighted tight Spectral Tetris fusion frame $\left(W_i\right)_{i=1}^D$ for $\mathcal{H}_M$ with dim$W_i=d_i$ for $i=1,\dots,D,$ if and only if $\left(\mbox{dim }V_i\right)_{i=1}^t \succeq \left(d_i \right)_{i=1}^D$.
\end{theorem}

Tight fusion frames are ideal in numerous applications of distributed processing because they are robust against additive noise and erasures. Also the fusion frame operator of a tight fusion frame is ideal for reconstruction purposes because it is a sequence of orthogonal projection operators which sum to a scalar multiple of the identity operator. Moreover, tight fusion frames are maximally robust against the loss of a single projection precisely when the tight fusion frame's projection operators are equidimensional, which is exactly the type of fusion frame Theorem \ref{tightff} constructs. Hence, the complete characterization of unit weighted tight fusion frames in Theorem \ref{tightff} is beneficial because this way researchers will know exactly when and how UFF can construct the tight fusion frames needed for their research. 

\section{Generalized Spectral Tetris Fusion Frame Constructions}\label{gffsec}

Given a spectrum for a desired fusion frame operator and dimensions for the subspaces we have seen, in Section \ref{pnstcff}, an easily implementable construction technique used to construct a unit weighted fusion frame with these properties. However, not all fusion frames are unit-weighted and so we would like a similar construction technique which allows for a fusion frame with prescribed weights. In \cite{CP}, they developed the first construction method for fusion frames with prescribed weights through an adapted version of Spectral Tetris. Moreover, they provide necessary and sufficient conditions for when a desired fusion frame can be constructed using Spectral Tetris. Hence they completely characterize Spectral Tetris fusion frames. This was quite an accomplishment considering the vast amount of literature on fusion frames and their need in application. 

In \cite{CP} they generalize PNSTC in such a way to construct Spectral Tetris fusion frames. However, due to the fact that PNSTC outputs conventional frames, we need a connection between these frames and fusion frames. From previous discussions in Section \ref{ffsec}, we have seen that a fusion frame $\{\left(W_i,w_i\right)\}_{i=1}^D$, with frame operator $\widetilde{S}$, arises from a conventional frame when we look at orthonormal bases $\left(\psi_{i,j}\right)_{j=1}^{d_i}$ of the fusion frame subspaces $W_i$. If we further assume that all subspaces have weight one, i.e. $w_i=1$ for all $i=\{1,\dots,D\}$, then $\{\psi_{i,j}\}_{i=1,j=1}^{D,d_i}$ is a frame with unit-norm vectors and has frame operator $S= \widetilde{S}$. This relationship lead to the definition of a {\it Spectral Tetris fusion frame} as defined in Definition \ref{stffdef}. 

However, all of the Spectral Tetris fusion frames constructed thus far have been unit weighted, and as such, in order to construct arbitrarily weighted Spectral Tetris fusion frames then this definition and connection between a fusion frame and a conventional frame needs to be generalized to be of use. In fact, we need to look at a tight frame for each subspace of a fusion frame instead of an orthonormal basis for each subspace, in order to identify a non-unit weighted fusion frame with a conventional frame. 

For a fusion frame $\{\left(W_i,w_i\right)\}_{i=1}^D$ in $\mathcal{H}_M$, recall our fusion frame operator $\widetilde{S}x=\sum_{i=1}^Dw_i^2\left(P_i\left(x \right) \right)$ for any $x\in \mathcal{H}_M$. Let $\{f_{i,j}\}_{i=1}^{d_i}$ be a tight frame for $W_i$ with frame operator $S$ and let $P_i$ be the orthogonal projection onto $W_i$. Then the fusion frame operator becomes:
\[\widetilde{S}x=\sum_{i=1}^Dw_i^2\left(P_i\left(x \right) \right)=\sum_{i=1}^D\sum_{j=1}^{d_i}\langle P_i\left(x \right), f_{i,j}\rangle f_{i,j}=\sum_{i=1}^D\sum_{j=1}^{d_i}\langle x, f_{i,j}\rangle f_{i,j}=Sx.\]
Hence, a non-unit weighted fusion frame arises from a conventional frame by identifying a tight frame for each subspace of the fusion frame. This is stated formally in the following theorem.

\begin{theorem}\label{fffthm}
For $i\in \{1,\dots,D\}$, let $w_i>0$, $W_i$ be a subspace of $\mathcal{H}_M$ and $\{f_{i,j}\}_{j=1}^{d_i}$ be a tight frame for $W_i$ with tight frame bounds $w_i^2$. Then the following are equivalent.
\begin{enumerate}
\item $\{\left(W_i,w_i\right)\}_{i=1}^D$ is a fusion frame whose fusion frame operator has spectrum $\{\lambda_m\}_{m=1}^M$.
\item $\{f_{i,j}\}_{i=1,j=1}^{D,d_i}$ is a frame whose frame operator has spectrum $\{\lambda_m\}_{m=1}^M$.
\end{enumerate}
\end{theorem}

Due to this relationship, to construct arbitrarily weighted fusion frames via Spectral Tetris, we will first construct a Spectral Tetris frame and then partition it in such a way so that the corresponding partition is a tight frame for each subspace of the fusion frame.

\begin{definition}
 Suppose $\{\left(W_i,w_i\right)\}_{i=1}^D$ is a fusion frame with fusion frame operator $\widetilde{S}$. We say $\{\left(W_i,w_i\right)\}_{i=1}^D$ is a {\it Spectral Tetris fusion frame} if there exists a Spectral Tetris frame $F = \{f_n\}_{n=1}^N$ with frame operator $S$ and a partition $\{J_i\}_{i=1}^D$ of $\{1,\dots,N\}$ such that $\{f_n\}_{n\in J_i}$ is a tight frame for $W_i$ with tight frame bound $w_i^2$. Further, we say $F$ and $\{J_i\}_{i=1}^D$ generate $\{\left(W_i,w_i\right)\}_{i=1}^D$.
\end{definition}

Since every fusion frame arises from a partition of a traditional frame, we introduce additional notation to easily identify subfamilies of frame vectors. Given a frame $F = \{f_n\}_{n=1}^N$ and a subset $J \subseteq \{1,\dots,N\}$, we denote the subfamily $F_J:=\{f_n| n \in J\}$. Since $F_J$ is a frame for its span, we again do not distinguish this set from its induced synthesis matrix.

We would like to construct a fusion frame which has a desired sequence of eigenvalues and desired subspace weights and dimensions. However, we make no mention of the norms of the vectors which span the subspaces of our fusion frame. In fact, although the eigenvalues, weights and dimensions are fixed for our fusion frame, the vector norms can vary. Moreover, it is possible for different sequences of norms to produce the same fusion frame. We illustrate this in the following example:

\begin{example} \cite{CP}
Consider $\mathcal{H}_2$ and a sequence of weights $\{\sqrt{2},1\}$ with corresponding subspace dimensions $\{2, 1\}$. Also, suppose we want the fusion frame operator to have eigenvalues $\{2, 3\}$. We use PNSTC to produce a variety of frames whose frame operator has this spectrum:
\begin{enumerate}
\item The sequence of norms $\{\sqrt{2},\sqrt{2},1\}$ produces the frame
$$\left[\begin{array}{ccc}
f_1&f_2&f_3
\end{array}\right]
=
\left[\begin{array}{ccc}
\sqrt{2}&0&0\\
0&\sqrt{2}&1
\end{array}\right].$$

\item The sequence of norms $\left(1,1,\sqrt{2},1\right)$ produces the frame

$$\left[\begin{array}{cccc}
g_1&g_2&g_3&g_4
\end{array}\right]
=
\left[\begin{array}{cccc}
1&1&0&0\\
0&0&\sqrt{2}&1
\end{array}\right].$$

\item The sequence of norms $\left( 1, \sqrt{\frac{3}{2}},\sqrt{\frac{3}{2}},1\right)$ produces the frame

$$\left[\begin{array}{cccc}
h_1&h_2&h_3&h_4
\end{array}\right]
=
\left[\begin{array}{cccc}
1&\sqrt{\frac{1}{2}}&\sqrt{\frac{1}{2}}&0\\
0&1&-1&1
\end{array}\right].$$
\end{enumerate}

A fusion frame $\{\left(W_i,w_i\right)\}_{i=1}^2,$ with weights $w_1=\sqrt{2}, w_2=1,$ is then obtained via PNSTC by defining $W_1=$span$\left(f_1,f_2\right)$, $W_2=$span$\left(f_3\right)$, or $W_1=$span$\left(g_1,g_2,g_3\right)$, $W_2=$span$\left(g_4\right)$, or $W_1=$span$\left(h_1,h_2,h_3\right)$, $W_1=$span$\left(h_4\right)$. All three generate the same fusion frame.
\end{example}

The differences among the constructions in this example are superficial; (2) simply splits a vector from (1) into two colinear vectors, and (3) takes two orthogonal vectors from (2) and combines them into a $2 \times 2$ block spanning the same 2-dimensional space. In fact, all Spectral Tetris frames which generate a given fusion frame are related in this manner. We state this more formally in the following theorem.

\begin{theorem}\cite{CP}\label{partitionthm}
If $\{\left(W_i, w_i\right)\}_{i=1}^D$ is a spectral tetris fusion frame in $\mathcal{H}_M$, then there exists a spectral tetris frame $F=\{f_n\}_{n=1}^N$ and a partition $\{J_i\}_{i=1}^D$ of $\{1,\dots,N\}$ generating this fusion frame such that $\|f_n\|=w_i$ and $\langle f_n,f_{n'}\rangle=0$, for $n,n' \in J_i$ for each $i \in \{1,\dots,D\}$.
\end{theorem}

From Theorem \ref{partitionthm}, we see that every Spectral Tetris fusion frame can be generated by a Spectral Tetris frame, where each subspace of the fusion frame is spanned by equal norm, orthogonal frame vectors. Moreover, the weights of the subspaces of the Spectral Tetris fusion frame are the norms of the frame vectors from the Spectral Tetris frame. With this relationship in mind, we are now ready to give necessary and sufficient conditions for constructing fusion frames via Spectral Tetris.

\begin{theorem}\label{thm3.4}\cite{CP}
Let $\{w_i\}_{i=1}^D$ be a sequence of weights, $\{\lambda_m\}_{m=1}^M$ a sequence of eigenvalues, and $\{d_i\}_{i=1}^D$ a sequence of dimensions. Let $N = \sum_{i=1}^D d_i$, and now consider each $w_i$ repeated $d_i$ times. We will use a double index to reference specific weights and a single index to emphasize the ordering:
\[ \{w_{i,j}\}_{i=1,j=1}^{D,d_i} = \{w_n\}_{n=1}^N.\]
Then Spectral Tetris can construct a fusion frame whose subspaces have the given weights and dimensions, and whose frame operator has the given spectrum if and only if there exists a Spectral Tetris ready (as in Definition \ref{def3.4}) permutation of $\{w_n\}_{n=1}^N$ and $\{\lambda_m\}_{m=1}^M$, say $\{w_{\sigma n}\}_{n=1}^N$ and $\{\lambda_{\sigma ' m}\}_{m=1}^M$ whose associated partition $1 \leq n_1\leq \cdots \leq n_M=N$ satisfies:

\begin{enumerate} 
\item $\sum_{n=1}^{n_i}w_{\sigma n}^2 < \sum_{m=1}^i\lambda_{\sigma ' m}$, then
		\begin{enumerate}
		\item if $\sum_{n=1}^{n_{i+1}}w_{\sigma n}^2 < \sum_{m=1}^{i+1} \lambda_{\sigma ' m}$, then for $w_{u,v},w_{p,q} \in \{w_{\sigma n}\}_{n=n_i}^{n_{i+1}+1}, v \neq q$
		
		\item if $\sum_{n=1}^{n_{i+1}}w_{\sigma n}^2 = \sum_{m=1}^{i+1} \lambda_{\sigma ' m}$, then for $w_{u,v},w_{p,q} \in \{w_{\sigma n}\}_{n=n_i}^{n_{i+1}}, v \neq q$
		\end{enumerate}
\item  $\sum_{n=1}^{n_i}w_{\sigma n}^2 = \sum_{m=1}^i\lambda_{\sigma ' m}$, then
		\begin{enumerate}
		\item if $\sum_{n=1}^{n_{i+1}}w_{\sigma n}^2 < \sum_{m=1}^{i+1} \lambda_{\sigma ' m}$, then for $w_{u,v},w_{p,q} \in \{w_{\sigma n}\}_{n=n_i+1}^{n_{i+1}+1}, v \neq q$
		\item if $\sum_{n=1}^{n_{i+1}}w_{\sigma n}^2 = \sum_{m=1}^{i+1} \lambda_{\sigma ' m}$, then for $w_{u,v},w_{p,q} \in \{w_{\sigma n}\}_{n=n_i+1}^{n_{i+1}}, v \neq q$
		\end{enumerate}
\end{enumerate}

for all $i=1,\dots, M-1$.
\end{theorem}

Theorem \ref{thm3.4} gives necessary and sufficient conditions for the construction of a Spectral Tetris fusion frame. With that said, it is possible for a sequence of weights/norms and a sequence of eigenvalues to satisfy the Spectral Tetris ready condition but no partition of these sequences satisfies the orthogonality conditions (1)(a,b) and (2)(a,b) of Theorem \ref{thm3.4}. Thus there exists a Spectral Tetris frame but there cannot exist a Spectral Tetris fusion frame. However this does not suggest that such a fusion frame cannot exist, it just implies that Spectral Tetris cannot construct such a fusion frame. We illustrate this in the following example.

\begin{example} 
Given the dimensions $\{d_i\}_{i=1}^5 = \{4,2,2,2,1\}$ and the eigenvalues $\{\lambda_m\}_{m=1}^6=\{\frac{11}{6},\frac{11}{6},\frac{11}{6},\frac{11}{6},\frac{11}{6},\frac{11}{6}\}$, PNSTC will construct the following unit norm frame with these properties.
$$\left[\begin{array}{ccccccccccc}
1&\sqrt{\frac{5}{12}}&\sqrt{\frac{5}{12}}&0&0&0&0&0&0&0&0\\
0&\sqrt{\frac{7}{12}}&-\sqrt{\frac{7}{12}}&\sqrt{\frac{1}{3}}&\sqrt{\frac{1}{3}}&0&0&0&0&0&0\\
0&0&0&\sqrt{\frac{2}{3}}&-\sqrt{\frac{2}{3}}&\sqrt{\frac{1}{4}}&\sqrt{\frac{1}{4}}&0&0&0&0\\
0&0&0&0&0&\sqrt{\frac{3}{4}}&-\sqrt{\frac{3}{4}}&\sqrt{\frac{1}{6}}&\sqrt{\frac{1}{6}}&0&0\\
0&0&0&0&0&0&0&\sqrt{\frac{5}{6}}&-\sqrt{\frac{5}{6}}&\sqrt{\frac{1}{12}}&\sqrt{\frac{5}{12}}\\
0&0&0&0&0&0&0&0&0&\sqrt{\frac{11}{12}}&-\sqrt{\frac{11}{12}}
\end{array}\right]$$

A unit weighted fusion frame with these dimensions and spectrum is known to exist due to combinatorial arguments. However, the hypotheses of Theorem \ref{thm3.4} cannot be satisfied in this case because no four columns can be chosen to be pairwise orthogonal. Hence there is no Spectral Tetris fusion frame with these properties.
\end{example}

We see from Theorem \ref{thm3.4} that given a spectrum for a fusion frame operator and a sequence of weights repeated appropriately for subspace dimensions, if these sequences are Spectral Tetris ready then PNSTC can construct a Spectral Tetris frame. If these sequences further satisfy conditions (1)(a,b) and (2)(a,b), then Theorem \ref{thm3.4} guarantees that the Spectral Tetris frame can be grouped into orthogonal, equal norm spanning sets for the subspaces. However, to determine if Spectral Tetris can be used we need to not only check if our sequences are Spectral Tetris ready but also that this permutation satisfies conditions (1)(a,b) and (2)(a,b). In general, this can be a very time consuming and tedious task. To alleviate this work, in \cite{CP} they provide several special cases of Theorem \ref{thm3.4} in which an appropriate ordering may be found much more easily.

\begin{proposition}\label{prop4.1}\cite{CP}
Given a sequence of norms $\{a_n\}_{n=1}^N$ and a sequence of eigenvalues $\{\lambda_m\}_{m=1}^M$ where $\sum_{n=1}^N a_n^2 = \sum_{m=1}^M \lambda_m$, if
\[\mbox{max }_{i,j \in \{1,\dots,N\}} \left( a_i^2+a_j^2\right) \leq \mbox{ min }_{m \in \{1,\dots,M\}}\lambda_m,\]
then the sequences can be made Spectral Tetris ready by systematically switching adjacent weights.
\end{proposition}

Proposition \ref{prop4.1} suggests that by possibly switching adjacent norms, the vector norms and eigenvalue sequences can become Spectral Tetris ready. In particular, this allows PNSTC to construct a frame with sequences which are not Spectral Tetris ready. The following example demonstrates this process.

\begin{example}\label{strex}
We will construct a Spectral Tetris frame on a sequence of vector norms and eigenvalues which are not Spectral Tetris ready. In $\mathcal{H}_2$, we will construct a 6-element frame with the sequence of norms $\{a_n\}_{n=1}^6=\{\sqrt{3}, 2,\sqrt{3},1,2,\sqrt{2}\}$ and eigenvalues $\{\lambda_m\}_{m=1}^2=\{9,8\}$.

First note that these sequences are not Spectral Tetris ready in the current order. Also, $\sum_{n=1}^6a_n^2=17=\sum_{m=1}^2\lambda_m$ and 
\[\mbox{max }_{i,j \in \{1,\dots,6\}} \left( a_i^2+a_j^2\right) = 8 \leq 8 \mbox{ min }_{m \in \{1,\dots, 2\}}\lambda_m,\]
hence by Proposition \ref{prop4.1} these sequences can be made Spectral Tetris ready by switching adjacent weights/norms. Therefore we can construct such a frame by using PNSTC and possibly switching norms. 

Starting the PNSTC construction of this frame yields: $f_1:= \sqrt{3}\cdot e_1$ and $f_2:= 2\cdot e_1$. Now, for row one we have a weight of $9-\left(\sqrt{3}\right)^2-\left(2\right)^2= 2$ left to add. Now $\lambda=2 \not\geq a_3^2=3$; hence in PNSTC we would typically add a $2 \times 2$ submatrix next. However, $a_4^2 = 1 < x=2 < 3=a_3^2$ and by Lemma \ref{pnstclem}, such a $2 \times 2$ submatrix does not exist. But, if we switch $a_3$ and $a_4$  to now have the vector norm order $\{\sqrt{3}, 2,1, \sqrt{3},2,\sqrt{2}\}$ then we have that $\lambda=2\geq 1=a_3^2$ and hence $f_3:=1\cdot e_1$. Now we have a weight of $\lambda=1$ left to add into row one. Since $\lambda =1 \not\geq 3=a_4^2$ then we add a $2 \times 2$ submatrix to yield $f_4:=\sqrt{\frac{3}{5}}\cdot e_1 + \sqrt{\frac{12}{5}}\cdot e_2$ and $f_5:= \sqrt{\frac{2}{5}}\cdot e_1 - \sqrt{\frac{18}{5}}\cdot e_2$. Thus we have sufficient weight in row one. Now for row two we need to add $8- \left(\sqrt{\frac{12}{5}}\right)^2-\left(\sqrt{\frac{18}{5}}\right)^2 = 2$ and hence we let $f_6:=2\cdot e_2$. Thus we have constructed the frame:
$$\left[\begin{array}{cccccc}
\sqrt{3}&2&1&\sqrt{\frac{3}{5}}&\sqrt{\frac{2}{5}}&0\\
0&0&0&\sqrt{\frac{12}{5}}&-\sqrt{\frac{18}{5}}&2
\end{array}\right].$$
Notice this frame has orthogonal rows with the 

desired norms $\{a_n\}_{n=1}^6=\{\sqrt{3}, 2,1,\sqrt{3},2,\sqrt{2}\}$ and

eigenvalues $\{\lambda_m\}_{m=1}^2=\{9,8\}$.
\end{example}

As seen in Example \ref{strex}, Proposition \ref{prop4.1} is a modification of PNSTC which allows the algorithm to handle non-Spectral Tetris ready orderings. To use this re-ordering technique, we will simply insert Table 6 between lines ((1)(b)(ii)) and ((1)(b)(ii)(A)) of the PNSTC algorithm, at Table 2. We will define this re-ordering procedure Spectral Tetris Re-Ordering (STR). Through the use of STR, we no longer have to check if sequences are Spectral Tetris ready before using PNSTC, which alleviates a possibly very time consuming task. 

In Table 6 we provide the Spectral Tetris Re-Ordering algorithm from \cite{CP}, which can be used in conjunction with PNSTC, to allow for non-Spectral Tetris ready orderings of vector norms and spectrum.

\begin{table}[ht]
\centering
\framebox{
\begin{minipage}[ht]{4.8in}
\vspace*{0.3cm}
{\sc \underline{STR: Spectral Tetris Re-Ordering Procedure}}

\vspace*{0.4cm}

\noindent {\bf Parameters:}
\begin{itemize}
\item Dimension $M \in \mathbb{N}$.

\item Number of frame elements $N \in \mathbb{N}$.

\item Eigenvalues $\left(\lambda_m\right)_{m=1}^M$ and vector norms $\{a_n\}_{n=1}^N$ such that $\sum_{n=1}^Na_n^2 = \sum_{m=1}^M \lambda_m$ and max$_{i,j \in \{1,\dots,N\}}\left(a_i^2+a_j^2\right) \leq$ min$_{m \in \{1,\dots,M\}}\lambda_m$.
\end{itemize}

\noindent {\bf Algorithm:}
\begin{enumerate}
\item If $\lambda_m > a_{n+1}^2$, then
		\begin{enumerate}
		\item temp $:=a_n$.
		\item $a_{n+1}:=a_n$.
		\item $a_{n+1} :=$temp.
		\item Go to STC (1ai).
		\end{enumerate}
\item end.
\end{enumerate}

\end{minipage}
}
\vspace*{0.2cm}
\caption{Procedure for running STC on a non-spectral-tetris-ready ordering.}
\label{stralg}
\end{table}

Using SFR always results in a Spectral Tetris ready ordering of the vector norms and eigenvalues. Hence in Theorem \ref{thm3.4} our main constraint now is to ensure that our sequences satisfy the orthogonality conditions (1)(a,b) and (2)(a,b). Since these orthogonality constraints are tedious, we will now provide more easily checked sufficient conditions for when a Spectral Tetris fusion frame can be constructed. 

\begin{theorem}\label{thm4.2b}\cite{CP}
Consider $\mathcal{H}_M$ and a sequence of weights $w_1 \leq w_2 \leq \dots \leq w_D$ with corresponding subspace dimensions $\{d_i\}_{i=1}^D$, and a sequence of eigenvalues $\lambda_1 \leq \lambda_2 \leq \dots \leq \lambda_M$. Let the doubly indexed sequence $\{w_{i,j}\}_{i=1,j=1}^{D,d_i}$ represent $w_i$ each repeated $d_i$ times. Now PNSTC/STR will build a weighted fusion frame $\{\left(W_i,w_i\right)\}_{i=1}^D$, with $\mbox{dim}\left(W_i\right)=d_i$ and whose frame operator has the given spectrum if there exists an ordering $\{w_n\}_{n=1}^N$ of $\{w_{i,j}\}_{i=1,j=1}^{D,d_i}$ such that
\begin{enumerate}
\item $\sum_{n=1}^N w_n^2=\sum_{m=1}^M \lambda_m$
\item $w_{D-1,1}^2 + w_{D,1}^2 \leq \lambda_1$
\item If $w_l=w_{i,j},w_{l'}=w_{i',j'}$ with $i=i'$ and $l < l'$, then $\sum_{n=l}^{l'-1}w_n^2 \geq 2\lambda_M$.
\end{enumerate}
\end{theorem}

Although the conditions in Theorem \ref{thm4.2b} are more relaxed than that of Theorem \ref{thm3.4}, finding an ordering of weights which achieves condition (3) is no small task. Intuitively, we would want to space like-weights as far apart as possible in our ordering in order to maximize $\sum_{n=l}^{l'-1}w_n^2$. When all of the subspaces have the same dimension then the ordering of the like-weights becomes obvious. We will start in this more obvious case and provide sufficient conditions for when PNSTC/STR can construct an equidimensional, tight fusion frame.

\begin{corollary}\cite{CP}\label{cormtb}
Consider $\mathcal{H}_M$ and a sequence of weights $w_1\leq w_2 \leq \dots \leq w_D$. PNSTC/STR can construct a tight weighted fusion frame with the given weights, all subspaces of dimension $k$, (eigenvalue $\lambda=\frac{k}{M}\sum_{i=1}^Da_i^2$) provided both of the following hold:
\begin{enumerate}
\item $w_{D-1}^2+w_D^2 \leq \lambda$
\item $\frac{k}{M} \leq \frac{1}{2}$
\end{enumerate}
\end{corollary}

Next, we specify Theorem \ref{thm4.2b} to the case of equidimensional fusion frames and provide sufficient conditions for when PNSTC/STR can construct such fusion frames. 

\begin{corollary}\label{cor4.4}\cite{CP}
Consider $\mathcal{H}_M$, a sequence of weights $w_1 \leq w_2 \leq \dots \leq w_D$ and a sequence of eigenvalues $\lambda_1 \leq \lambda_2 \leq \dots \leq \lambda_M$. PNSTC/STR can construct a weighted fusion frame $\{\left(W_i,w_i\right)\}_{i=1}^D$, all subspaces dimension $k$, and with the given spectrum provided all of the following hold:
\begin{enumerate}
\item $k\sum_{n=1}^D w_n^2 = \sum_{m=1}^M \lambda_m$
\item $w_{D-1}^2 +w_D^2 \leq \lambda_1$
\item $\sum_{n=1}^Dw_n^2 \geq 2\lambda_M$
\end{enumerate}
\end{corollary}

\begin{remark}
To construct the fusion frame in Corollary \ref{cormtb} and Corollary \ref{cor4.4}, write each weight $w_i$ repeated $k$ times and arrange these weights as follows: $\{a_1,\dots,a_m,a_1,\dots,a_m,\dots\}$. Then proceed to use PNSTC/STR on this collection of norms. We provide this arrangement of the sequence of weights because it guarantees that such a fusion frame can be constructed so long as all of the conditions in the respective corollary are met. However, other arrangements are possible.
\end{remark}

Next we provide an example to help illustrate the Spectral Tetris construction of an equidimensional, weighted fusion frame and to do this we will utilize Corollary \ref{cor4.4}. 

\begin{example}
In $\mathcal{H}_5$ we will construct a weighted fusion frame with 9 two-dimensional subspaces with weights $\{w_i\}_{i=1}^{9}=\{1,1,1,1,\sqrt{2},\sqrt{2},\sqrt{3},\sqrt{3},2\}$ respectively and spectrum $\{\lambda_m\}_{m=1}^5=\{7,7,7,7,8\}$. Notice that conditions (1), (2), and (3) of Corollary \ref{cor4.4} are met. Indeed:
\begin{enumerate}
\item $k\sum_{i=1}^9w_i^2=2\left(18\right)=36=\sum_{i=1}^5\lambda_m$
\item $w_{D-1}^2+w_D^2=7\leq 7=\lambda_1$
\item $\sum_{i=1}^9w_i^2=18\geq 16=2\lambda_5$
\end{enumerate}
and hence such a construction is possible. 

To construct our fusion frame we first construct the corresponding Spectral Tetris frame via PNSTC. First we write each norm $k$ times and arrange these weights in the following order:
\[\{1,1,1,1,\sqrt{2},\sqrt{2},\sqrt{3},\sqrt{3},2,1,1,1,1,\sqrt{2},\sqrt{2},\sqrt{3},\sqrt{3},2\}.\]

Using PNSTC we will construct the following Spectral Tetris frame:
$$\left[\begin{array}{cccccccccccccccccc}
1&1&1&1&\sqrt{2}&\sqrt{\frac{2}{3}}&\sqrt{\frac{1}{3}}&0&0&0&0&0&0&0&0&0&0&0\\
0&0&0&0&0&\sqrt{\frac{4}{3}}&-\sqrt{\frac{8}{3}}&\sqrt{3}&0&0&0&0&0&0&0&0&0&0\\
0&0&0&0&0&0&0&0&2&1&1&1&0&0&0&0&0&0\\
0&0&0&0&0&0&0&0&0&0&0&0&1&\sqrt{2}&\sqrt{2}&1&1&0\\
0&0&0&0&0&0&0&0&0&0&0&0&0&0&0&\sqrt{2}&-\sqrt{2}&2
\end{array}\right].$$

Next, grouping the frame vectors in the following way will yield our desired fusion frame:
\[W_1=\mbox{ span}\{f_1,f_{10}\}; W_2=\mbox{ span}\{f_2,f_{11}\};W_3=\mbox{ span}\{f_3,f_{12}\};\]
\[W_4=\mbox{ span}\{f_4,f_{13}\};W_5=\mbox{ span}\{f_5,f_{14}\};W_6=\mbox{ span}\{f_6,f_{15}\};\]
\[W_7=\mbox{ span}\{f_7,f_{16}\};W_8=\mbox{ span}\{f_8,f_{17}\};W_9=\mbox{ span}\{f_9,f_{18}\};\]
where each subspace is two-dimensional, has the respective desired weight and the spectrum of the fusion frame operator is $\{7,7,7,7,8\}$ as desired. 
\end{example}

\begin{remark}\cite{CP}
In order for PNSTC to build a desired fusion frame, a complex relationship
among partial sums of weights, partial sums of eigenvalues, and dimensions of our subspaces must be satisfied according to Theorem \ref{thm3.4}. We simplified this relationship in Theorem \ref{thm4.2b} and its corollaries to achieve concrete constructions via PNSTC/STR. While these extra assumptions still allow a variety of fusion frames to be created, they are best suited for fusion frames with relatively 
flat spectrum. For example, (1) and (3) of Corollary \ref{cor4.4} imply
\[ \frac{\sum_{m=1}^M\lambda_m}{k} \geq 2\lambda_M,\]
and this can clearly be manipulated to 
\[ \frac{\mbox{Average }\left(\{\lambda_m\}_{m=1}^M\right)}{2\lambda_M} \geq \frac{k}{M}.\]

Hence if we desire PNSTC/STR to guarantee the construction of fusion frames with relatively large subspaces, our prescribed frame operator must have a relatively 
flat spectrum. However, the conditions used here are of the correct order for practical applications. That is, we generally do not work with large subspaces or with eigenvalues for the frame operator which are very spread out.
\end{remark}


\section{Concluding Remarks}

Before the development of Spectral Tetris there were no general methods for constructing frames nor fusion frames and this proved to be a problem since both frames and fusion frames are used in numerous applications where explicit frames with desired properties are necessary. In the present paper we have seen the complete development of Spectral Tetris. From being restricted to only constructing unit norm tight frames to eventually being generalized and completely characterized to construct frames with prescribed spectrum and vector norms. Then we proceeded to see the development of Spectral Tetris fusion frames and have found numerous classifications and construction techniques for constructing fusion frames based on Spectral Tetris frames. This ultimately developed into a complete characterization of Spectral Tetris fusion frames. Throughout we have not only given the complete characterization of Spectral Tetris frames and Spectral Tetris fusion frames, but we have given numerous easily check-able sufficient conditions for using Spectral Tetris. Also, for each construction method we have given an easily implementable algorithm and  numerous examples to help illustrate the constructions. Throughout our paper we only addressed finite frame constructions; however, infinite dimensional Spectral Tetris frame constructions do exist and have been studied in \cite{BJ}. Overall, we have seen that Spectral Tetris provides an easily implementable construction algorithm for sparse frames and  sparse fusion frames with desired properties.  



\end{document}